\documentclass[reqno]{amsart}
\usepackage{amssymb,xypic,mathrsfs,hyperref,comment,todonotes,fullpage,pagecolor}
\hypersetup{ 
colorlinks,
citecolor=black,
filecolor=blue,
linkcolor=black,
urlcolor=magenta
}

\newtheorem{theorem}{Theorem}
\newtheorem{lemma}[theorem]{Lemma}
\newtheorem{proposition}[theorem]{Proposition}
\newtheorem*{theoremI}{Theorem}
\newtheorem*{propositionI}{Proposition}
\newtheorem{corollary}[theorem]{Corollary}
\newtheorem*{corollaryI}{Corollary}
\newtheorem*{conjectureI}{Conjecture}

\newtheorem{question}[theorem]{Question}
\newtheorem{conjecture}[theorem]{Conjecture}

\theoremstyle{definition}
\newtheorem{definition}[theorem]{Definition}
\newtheorem{remark}[theorem]{Remark}

\numberwithin{equation}{section}
\numberwithin{theorem}{section}
\newcommand{\msc}[1]{\mathscr{#1}}
\newcommand{\mrm}[1]{\mathrm{#1}}
\newcommand{\mbb}[1]{\mathbb{#1}}

\newcommand{\ot}{\otimes}
\newcommand{\Hom}{\mathrm{Hom}}

\newcommand{\sHom}{\mathscr{H}om}
\newcommand{\sEnd}{\mathscr{E}nd}

\newcommand{\uHH}{\underline{H\! H}}
\newcommand{\HH}{H\! H}
\newcommand{\Ext}{\mathrm{Ext}}
\newcommand{\sExt}{\mathscr{E}xt}
\newcommand{\sRHom}{\mathrm{R}\mathscr{H}om}
\newcommand{\sREnd}{\mathrm{R}\mathscr{E}nd}

\newcommand{\Aut}{\mathrm{Aut}}
\newcommand{\Qcoh}{\mathrm{Qcoh}}
\renewcommand{\O}{\mathscr{O}}

\newcommand{\X}{\mathscr{X}}
\newcommand{\uSA}{\underline{\mathrm{SA}}}

\newcommand{\Spec}{\operatorname{Spec}}
\newcommand{\codim}{\mathrm{codim}}

\renewcommand{\hat}{\widehat}

\newcommand{\mfm}{\mathfrak{m}}
\newcommand{\End}{\operatorname{End}}
\newcommand{\rk}{\operatorname{rk}}

\newcommand{\gr}{\operatorname{gr}}

\newcommand{\pr}{\operatorname{pr}}
\newcommand{\DD}{\mathbb{D}}

\title[Hochschild cohomology of quotient orbifolds]{The Hochschild cohomology ring of a global quotient orbifold}

\author{Cris Negron}
\thanks{The first author was supported by NSF Postdoctoral Fellowship DMS-1503147}
\address{Department of Mathematics\\Massachusetts Institute of Technology\\
Cambridge, MA 02139, USA}
\email{negronc@mit.edu}

\author{Travis Schedler\\
with appendices by Pieter Belmans, Pavel Etingof and the authors}
\thanks{The second author was supported by NSF Grant DMS-1406553}
\address{Huxley 622\\Imperial College London\\South Kensington Campus\\London SW7 2AZ, UK}
\email{trasched@gmail.com}
\date{\today}

\begin{document}

\maketitle

\begin{abstract}
We study the cup product on the Hochschild cohomology of the stack quotient $[X/G]$ of a smooth quasi-projective variety $X$ by a finite group $G$.  More specifically, we construct a $G$-equivariant sheaf of graded algebras on $X$ whose $G$-invariant global sections recover the associated graded algebra of the Hochschild cohomology of $[X/G]$, under a natural filtration.  This sheaf is an algebra over the polyvector fields $T^{poly}_X$ on $X$, and is generated as a $T^{poly}_X$-algebra by the sum of the determinants $\det(N_{X^g})$ of the normal bundles of the fixed loci in $X$.  
% We completely describe the multiplication on this sum of determinants in terms of the geometry of the fixed loci. 
 We employ our understanding of Hochschild cohomology to conclude that the analog of Kontsevich's formality theorem, for the cup product, does not hold for Deligne--Mumford stacks in general. We discuss
relationships with orbifold cohomology, extending Ruan's cohomological conjectures.  This employs a trivialization of the determinants in the case of a symplectic group action on a symplectic variety $X$, 
which requires (for the cup product) a nontrivial normalization missing in previous literature.
\end{abstract}

\section{Introduction}

Fix a base field $k$ of characteristic $0$ (or of a prime
characteristic $p$ which is larger than the dimension of the variety $X$ and
 does not divide the size of the group $G$).  All
varieties will be quasi-projective over $k$.  We often take
$X^2=X\times X$ for $X$ a variety, or stack.
\par

In this paper we describe the multiplicative structure on the Hochschild cohomology $\HH^\bullet([X/G])$ of the quotient stack of a smooth quasi-projective variety $X$ by a finite group $G$, which is assumed to act by automorphisms over $\operatorname{Spec} k$.  We show first that the multiplication on Hochschild cohomology can be understood in terms of the geometry of $X$ as a $G$-variety, then proceed to discuss a number of applications to formality questions and orbifold cohomology.  Such an object $[X/G]$ is also known as a global quotient orbifold.
\par

Our description of the cup product can be seen as a version of the Hochschild-Kostant-Rosenberg theorem for such
quotient stacks, taking into account the multiplicative structure. In
its simplest form, for smooth affine varieties $X$, this result states
that $\HH^\bullet(X) \cong \wedge_{\mathscr{O}(X)}^\bullet T(X)$,
called polyvector fields on $X$~\cite{HKR} (as algebras), and it has
been generalized for example to the global setting in~\cite{yekutieli02,calaquevandenbergh10}.  There are many applications
of understanding this together with its algebraic structure, such as
to deformation theory.  In the case at hand of a global quotient orbifold $[X/G]$, the
Hochschild cohomology controls the deformation theory of the $G$-equivariant
geometry of $X$ (see Section~\ref{sect:remdefo} below).
\par

As a consequence of our description of the cup product we find that the multiplicative form of Kontsevich's formality theorem does not hold for smooth Deligne-Mumford stacks in general.  We also observe a vector space identification between a variant of the Hochschild cohomology and the orbifold cohomology of $[X/G]$, in the particular instance in which $X$ is symplectic and $G$ acts by symplectic automorphisms.  When $X$ is projective, in addition to being symplectic, we conjecture that there is furthermore an {\it algebra} isomorphism $\HH^\bullet([X/G])\cong H^\bullet_\mrm{orb}([X/G])$.  When the quotient variety $X/G$ admits a crepant resolution this conjectural algebra isomorphism reduces to Ruan's Cohomological Hyperk\"ahler Resolution Conjecture, and hence provides a resolution free variant of Ruan's conjecture.

We elaborate on our description of Hochschild cohomology, and subsequent applications, below.  For the remainder of the introduction take $\X=[X/G]$, for $X$ a smooth quasi-projective variety and $G$ a finite group.

\subsection{The structure of Hochschild cohomology}

For us, the Hochschild cohomology of $\mathscr{X}$ will be the algebra of self-extensions
\[
\HH^\bullet(\X)=\Ext^\bullet_{\X^2}(\Delta_\ast\O_\X,\Delta_\ast\O_\X).
\]
Motivated by methods from noncommutative algebra, we show that the cohomology can be given as an algebra of extensions between the pushforward $\Delta_\ast\O_X$ on $X^2$ and a certain smash product construction $\Delta_\ast\O_X\rtimes G$.  Specifically, we explain below that the graded group of extensions $\Ext^\bullet_{X^2}(\Delta_\ast\O_X,\Delta_\ast\O_X\rtimes G)$ is naturally a graded algebra with an action of $G$ by algebra automorphisms.  We provide further an algebra identification $\HH^\bullet(\X)=\Ext^\bullet_{X^2}(\Delta_\ast\O_X,\Delta_\ast\O_X\rtimes G)^G$ (see Section~\ref{sect:loc}).
\par

Our Hochschild cohomology with coefficients in the smash product has an obvious sheaf analog.  For $\pi: X^2 \to X$ the first projection, we take
\[
\uHH^\bullet(\msc{X})=\pi_* \sExt^\bullet_{X^2}(\Delta_* \O_X,\Delta_* \O_X\rtimes G),
\]
which we will refer to as the \emph{local equivariant Hochschild cohomology} of $\msc{X}$.
We will explain that the ``sheaf-valued'' cohomology $\uHH^\bullet(\msc{X})$ is naturally a $G$-equivariant sheaf of graded algebras on $X$, i.e.\ a sheaf of algebras on $\msc{X}$ (see Remark~\ref{rem:intrinsic}), and that the usual local-to-global spectral sequence 
\begin{equation}\label{eq:105}
E_2=
H^\bullet(X,\uHH^\bullet(\msc{X}))^G\Rightarrow \HH^\bullet(\mathscr{X})
\end{equation}
is multiplicative.  By \cite{arinkinetal14}, this spectral sequence
degenerates at the $E_2$-page (see Section~\ref{ss:form-deg}).  In this sense, the hypercohomology of $\uHH^\bullet(\msc{X})$ provides an approximation of the Hochschild cohomology ring of the stack quotient.
\par

Our first main result can be summarized as follows. For every scheme $Y$, let $T^{poly}_{Y} := \bigwedge_{\O_Y} T_Y$ be the sheaf of polyvector fields on $Y$.

\begin{theorem}\label{thm:intro}
Let $G$ be a finite group acting on a smooth quasi-projective variety $X$, and consider the quotient orbifold $\msc{X}=[X/G]$.  Then there is a canonical identification of equivariant sheaves
\begin{equation}\label{eq:166}
\uHH^\bullet(\msc{X})=\bigoplus_{g\in G}T^{poly}_{X^g}\ot_{\O_X}\det(N_{X^g}).
\end{equation}
The induced algebra structure on the sum $\oplus_{g\in G} T^{poly}_{X^g}\otimes_{\O_X} \det(N_{X^g})$ can be understood as follows:
\begin{enumerate}
\item[(i)] The identity component $T^{poly}_X$ is in the center of $\uHH^\bullet(\msc{X})$.
\item[(ii)] As a $T^{poly}_X$-algebra, $\uHH^\bullet(\msc{X})$ is generated by the determinants of the normal bundles $\oplus_{g\in G}\det(N_{X^g})$.
\item[(iii)] The sum of the determinants themselves form a $G$-graded, equivariant subalgebra in $\uHH^\bullet(\msc{X})$.
\item[(iv)] Each multiplication
\[
m:\det(N_{X^g})\otimes_{\O_X} \det(N_{X^h})\to \det(N_{X^{gh}})
\]
is nonvanishing exactly on those components $Y$ of the intersection $X^g\cap X^h$ which are also components of $X^{gh}$, and along which $X^g$ and $X^h$ intersect transversely.
\end{enumerate}
Furthermore, there is a filtration on the global cohomology $\HH^\bullet(\msc{X})$ under which we have an identification of algebras
\begin{equation}\label{eq:151}
\gr \HH^\bullet(\X)=H^\bullet\left(X,\oplus_{g\in G}\!\ T^{poly}_{X^g}\otimes_{\O_X} \det(N_{X^g})\right)^G.
\end{equation}
\end{theorem}

In the unpublished manuscript~\cite{anno}, a version of the theorem above
appears in the affine setting (in terms of the obvious explicit
formula for cup product, rather than (i)--(iv); see Section~\ref{sect:altprod}).  We note that the {\it vector space} identification~\eqref{eq:151} is due to Ginzburg-Kaledin in the affine case, and Arinkin--C\u{a}ld\u{a}raru--Hablicsek in general~\cite{arinkinetal14}.
\par

%We describe the subalgebra alluded to in points (iii) and (iv) in more detail below.
The subalgebra of (iii) is denoted $\uSA(\msc{X}):=\oplus_{g\in G}\det(N_{X^g})$ throughout.  We see from (ii), or more precisely Theorem~\ref{thm:Tact} below, that this subalgebra provides all of the novel features of the Hochschild cohomology of the orbifold $\msc{X}$, as compared with that of a smooth variety.  Furthermore, as the points of the fixed spaces $X^g$ account for the points of $\msc{X}$ which admit automorphisms, this algebra also represents a direct contribution of the stacky points of $\msc{X}$ to the Hochschild cohomology.

An explicit description of the multiplicative structure on the subalgebra $\uSA(\msc{X})$ is given in Theorem~\ref{thm:SA}, and the multiplicative structure on the entire cohomology $\uHH^\bullet(\msc{X})$ is subsequently obtained from Theorem~\ref{thm:Tact}.  The algebra identification~\eqref{eq:151} appears in Corollary~\ref{cor:arinkinetal_mult}.

\begin{remark}
Since $\uSA(\msc{X})$ is an equivariant sheaf of algebras on $X$, we may view it as a sheaf of algebras on the quotient $\msc{X}$.  However, the reader should be aware that the algebra $\uSA(\msc{X})$ is not known to be an invariant of $\msc{X}$ at the moment.  So a notation $\uSA(X\to \msc{X})$ may be more appropriate here.
\end{remark}

Of course, in the affine case the derived global sections vanish in positive degree.  As a result, $\Gamma \uHH^\bullet(\msc{X}) = \HH^\bullet(\msc{X})$, and
we do not need a spectral sequence.
Hence Theorem~\ref{thm:intro}, in conjunction with the description of the algebra $\uSA(\msc{X})$ below, provides a complete description of the Hochschild cohomology in this case.
\begin{corollaryI}
Let $G$ be a finite group acting on a smooth affine variety $X$, and take $\X=[X/G]$.  There is an algebra identification
\begin{equation}\label{eq:151-nogr}
\HH^\bullet(\X)=\Gamma\left(X,\oplus_{g\in G}\!\ T^{poly}_{X^g}\otimes_{\O_X} \det(N_{X^g})\right)^G
\end{equation}
where the multiplication on the right hand side is as described above.
\end{corollaryI}
Note that there is no ``gr'' appearing in the corollary since the filtration is now trivial (it matches the cohomological grading).

The above corollary was established in work of Shepler and Witherspoon~\cite{SW} in the specific instance of a finite group acting linearly on affine space (cf.~\cite{anno}).  We note that the description of cohomology given in the Corollary, as the global sections of our algebra of polyvector fields, is not exclusive to affine varieties.  Indeed, we provide some projective examples in Section~\ref{sect:rem_deg} for which we have a similar equality~\eqref{eq:151-nogr}.

\subsection{Formality, symplectic quotients, and orbifold cohomology}

We say that the Hochschild cohomology of a Deligne--Mumford stack $\msc{X}$ is formal if the sheaf-extensions $\msc{E}xt_\msc{X}^\bullet(\Delta_\ast\O_\msc{X},\Delta_\ast\O_\msc{X})$ and the derived endomorphisms $\mrm{R}\msc{H}om_\msc{X}(\Delta_\ast\O_\msc{X},\Delta_\ast\O_\msc{X})$ are connected by a sequence of $A_\infty$ quasi-isomorphisms (see Section~\ref{sect:formality}).  In characteristic zero, there is a highly nontrivial isomorphism between $T^{poly}_X$ and the derived endomorphisms of $\Delta_\ast\O_X$, for smooth $X$, using the composition of the Hochschild--Kostant--Rosenberg isomorphism with exterior multiplication by the square-root of the Todd class of $T_X$ (claimed on the level of global sections by Kontsevich~\cite[Claim, Section 8.4]{kontsevich03}, see also~\cite[Claim 5.1]{Cal-mp2}, and proved also by Calaque-Van den
    Bergh~\cite[Corollary 1.5]{calaquevandenbergh10}; the local statement is a consequence of \cite[Theorem 1.1]{calaquevandenbergh10II}).  We show that the analog of this result does not hold for projective Calabi--Yau quotient orbifolds with singular coarse space.

\begin{theoremI}[\ref{thm:informal}]
Suppose $X$ is projective and that $G$ is a finite subgroup in $\mrm{Aut}(X)$.  Suppose additionally that the action of $G$ is not free, and that the quotient $\msc{X}=[X/G]$ is Calabi--Yau. Then the Hochschild cohomology of the orbifold $\msc{X}$ is not formal.
\end{theoremI}

One can construct such Calabi--Yau quotients via symplectic varieties equipped with symplectic group actions.  Symplectic orbifolds feature prominently in the literature, particularly in studies of crepant resolutions and quantization.

\begin{corollaryI}[\ref{cor:symp_informal}]
Let $X$ be a projective symplectic variety and $G\subset \mrm{Sp}(X)$ be a finite subgroup for which the action is not free.  Then the Hochschild cohomology of the corresponding orbifold quotient $\msc{X}$ is not formal.
\end{corollaryI}

Kummer varieties, i.e.\ quotients of abelian surfaces by the $\mbb{Z}/2\mbb{Z}$-inversion action, provide explicit examples of $\msc{X}$ as in the corollary, and hence explicit examples of smooth Deligne--Mumford stacks with non-formal Hochschild cohomology.
\par

We continue our analysis of symplectic quotients in Section~\ref{sect:symplectic}, where we provide a simplification of Theorem~\ref{thm:intro} in the symplectic setting over $\mbb{C}$.  In this case one can produce simultaneous trivializions of the determinants of the normal bundles $\det(N_{X^g})$ which provide the following algebra identification.

\begin{theoremI}[\ref{thm:HHsymp}]
Suppose $X$ is a symplectic variety over $\mbb{C}$, and that $G$ acts on $X$ by symplectic automorphisms.  Then there is an algebra isomorphism
\begin{equation}\label{eq:267}
\uHH^\bullet(\msc{X})\cong \bigoplus_{g\in G} \Sigma^{-\codim(X^g)}T^{poly}_{X^g},
\end{equation}
where the (shifted) polyvector fields multiply in the expected manner.
\end{theoremI}

By the ``expected manner" we mean, vaguely, the multiplication one would expect given the vanishing result of Theorem~\ref{thm:intro} (iv) and the standard multiplication of polyvector fields.  A precise description of this multiplication is given in the discussion following Theorem~\ref{thm:HHsymp}.
\par

We note one {\it cannot} achieve the above isomorphism by simply employing the trivializations $\det(N_{X^g})\cong \O_{X^g}$ provided by the symplectic form on $X$.  One has to instead rescale these trivializations by the square roots of the determinants of the elements $(1-g)$, which we interpret as invertible operators on their respective normal bundles $N_{X^g}$.  In Appendix~\ref{sect:pav}, with Pavel Etingof, we show in the particular case of a finite group acting linearly on a symplectic vector space that the rescaled trivializations produce the desired isomorphism~\eqref{eq:267}.  For arbitrary $\msc{X}$ we reduce to the linear case to obtain the isomorphism~\eqref{eq:267}.  We note that the introduction of these scaling factors corrects a persistent error in the literature (see Section~\ref{sect:remark}).

For a symplectic quotient $\msc{X}$, we also propose a number of relations between Hochschild cohomology and orbifold cohomology $H^\bullet_\mrm{orb}(\msc{X})$ for $\msc{X}$.  In particular, we provide a linear identification between the Hochschild cohomology and the orbifold cohomology of a projective symplectic quotient orbifold at Corollary~\ref{cor:501}.  We conjecture that there is furthermore an algebra identification between these two cohomologies.

\begin{conjectureI}[\ref{conj:HH_H_orb}]
For any projective, symplectic, quotient orbifold $\msc{X}$ over $\mbb{C}$ there is an algebra isomorphism $\HH^\bullet(\msc{X})\cong H^\bullet_\mrm{orb}(\msc{X})$.
\end{conjectureI}

If the (generally singular) quotient variety $X/G$ admits a crepant resolution then our conjecture reduces to Ruan's Cohomological
Hyperk\"{a}hler Resolution Conjecture.  (See Section~\ref{sect:orb_prod}.)  As Ruan's conjecture relies on the existence of a crepant resolution, Conjecture~\ref{conj:HH_H_orb} represents a resolution free variant of the Cohomological
Hyperk\"{a}hler Resolution Conjecture.
%Generalizations of Corollary~\ref{cor:501} and Conjecture~\ref{conj:HH_H_orb} for arbitrary quasi-projective symplectic orbifolds are given in Theorem~\ref{thm:HHorb} and Conjecture~\ref{conj:575}.  Our investigations of orbifold cohomology are inspired by results of Ginzburg and Kaledin in the affine setting~\cite{ginzburgkaledin04}.

\begin{remark}
%In considering our analyses of Hochschild cohomology and orbifold cohomology for projective symplectic orbifolds $\msc{X}$ over $\mbb{C}$, one should n
Note that projective complex symplectic varieties (which pertain to the last two sections) are, from the differential geometric perspective, compact hyperk\"ahler manifolds.  
%The explicit identification of the two structures 
This is due to Yau's solution to the Calabi conjecture~\cite{yau78}. The
converse also holds under our quasi-projectivity assumption.
% (we only need to assume that the variety is analytic
% quasi-projective, due to Chow's theorem).
\end{remark}

\subsection{Generalities for the Hochschild cohomology of algebraic stacks}
\label{sect:remdefo}

In principle, for a general algebraic stack $\msc{X}$, the Hochschild cohomology $\HH^\bullet(\msc{X})$ defined via self-extensions of the pushforward $\Delta_\ast\O_\msc{X}$ need not be identified with the Hochschild
cohomology $\HH^\bullet_\mrm{ab}(\mrm{Qcoh}(\msc{X}))$ of the category of quasi-coherent sheaves on $\X$, as defined
in~\cite{lowenvandenbergh05}.  However, such an identification is desirable as the cohomology  $\HH^\bullet_\mrm{ab}(\mrm{Qcoh}(\msc{X}))$ has the correct theoretical interpretation, namely as the cohomology of the dg Lie algebra controlling the deformation theory of the category of sheaves on $\msc{X}$, while the self-extension algebra $\HH^\bullet(\msc{X})$ is more accessible to computation.
\par

In an appendix by Pieter Belmans, Appendix~\ref{sect:cat_coho}, it is shown that for many reasonable stacks the algebra of self-extensions $\HH^\bullet(\msc{X})$ does in fact agree with the Hochschild cohomology $\HH^\bullet_\mrm{ab}(\mrm{Qcoh}(\msc{X}))$ of the abelian category $\mrm{Qcoh}(\msc{X})$, as desired.

\begin{propositionI}[\ref{prop:belmans}]
  Let $\msc{X}$ be a perfect algebraic stack over
  $\operatorname{Spec}k$.  Then there exists an isomorphism of graded
  algebras $\HH^\bullet(\msc{X})\cong
  \HH^\bullet_\mrm{ab}(\mrm{Qcoh}(\msc{X}))$.  In particular, we have
  such an isomorphism for any global quotient orbifold $\msc{X}$.
\end{propositionI}

The above result implies that there is a natural graded Lie structure
on the self-extension algebra $\HH^\bullet(\msc{X})$.  Although we do
not pursue the topic here, one can attempt to construct a sufficiently
natural Lie algebra structure on the local Hochschild cohomology
$\uHH^\bullet(\msc{X})$ and ask whether Kontsevich formality for the
{\it Lie} structure on the Hochschild cohomology of $\msc{X}$ is
similarly obstructed, as in Theorem~\ref{thm:informal}.  However, we
expect that one has to at least push this sheaf
$\uHH^\bullet(\msc{X})$ forward to the coarse space $X/G$ in order
to construct such an appropriate Lie structure.  (Here, $X/G$ is
the categorical quotient scheme; in the case that $X$ is affine,
$X/G = \Spec \O(X)^G$.)

%We remark that,
\begin{remark}
If instead of considering the Hochschild cohomology of
the abelian category, one works in the setting of derived algebraic
geometry and takes the Hochschild cohomology of the DG category
$\operatorname{QC}(\X)$ of quasi-coherent sheaves on $\X$, then the
fact that one recovers self-extensions of $\Delta_\ast\O_\X$ is a direct consequence of
 \cite[Theorem 1.2.(2)]{BFN-dag}, which identifies endofunctors of
$\operatorname{QC}(\X)$ with quasi-coherent sheaves on $\X^2$, sending
the identity to $\Delta_\ast \O_\X$. For more details, see the last
two paragraphs of the proof of Proposition \ref{prop:belmans}.
\end{remark}

\subsection{References to related literature}

\begin{comment}
As stated previously, the sheaf identification~\eqref{eq:166} is due to Anno~\cite{anno}, and the vector space identification~\eqref{eq:151} appeared in work of Ginzburg--Kaledin and Arinkin--C\u{a}ld\u{a}raru--Hablicksek~\cite{ginzburgkaledin04,arinkinetal14}.  For the case of a finite group acting linearly on affine space, the cup product on Hochschild cohomology was described explicitly in work of Shepler and Witherspoon~\cite{SW}, and an alternate formula for the cup product was given by Anno in the general affine setting~\cite{anno}.  Our algebra $\uSA(\msc{X})$ is a generalization of the ``volume subalgebra" considered in~\cite{SW}.  We note that the material presented here relies on~\cite{arinkinetal14} only to establish degeneration of the spectral sequence~\eqref{eq:105}.
\end{comment}

In the case of a finite group $G$ acting linearly on a vector space $V$ (considered here as an affine space) the Gerstenhaber structure on Hochschild cohomology $\HH^\bullet([V/G])$ is understood.  A concise algebraic description of the Gerstenhaber bracket can be found in work of the first author and Witherspoon~\cite{negronwitherspoon17} (see also~\cite{SW2,halbouttang10}).  For more general orbifolds, even affine quotient orbifolds, an analogous description of the bracket has yet to appear.
\par

A general analysis of Hochschild homology and cohomology for orbifolds, from the perspective of the derived category, can be found in a sequence of papers by C\u{a}ld\u{a}raru and Willerton~\cite{caldararuI,Cal-mp2,caldararuwillerton10}.
\par

A $C^\infty$-analog of Theorem~\ref{thm:intro} appears in work of Pflaum, Posthuma, Tang, and Tseng~\cite{pflaumetal11}.  In~\cite{pflaumetal11} the authors consider arbitrary ($C^\infty$-)orbifolds.  Such objects may not be given as global quotients in general.  

\subsection{Organization of the paper}

In Section~\ref{sect:rem_deg} we discuss some projective examples.  Section~\ref{sect:bg} is dedicated to background material.  In Section~\ref{sect:ss} we introduce the smash product $\Delta_\ast\O_X\rtimes G$ and establish the necessary relationships between the local Hochschild cohomology $\uHH^\bullet(\msc{X})$ and the cohomology of the stack quotient $\X$.  In Section~\ref{sect:SPprelim} we establish some generic relations between normal bundles and tangent bundles for the fixed spaces, and in Sections~\ref{sect:A} and~\ref{sect:product} we use these relations to give a geometric description the Hochschild cohomology $\uHH^\bullet(\msc{X})$ as a sheaf of algebras on $\msc{X}$.
\par

We discuss formality and Calabi--Yau orbifolds in Section~\ref{sect:formality}.  In Section~\ref{sect:symplectic} we describe the Hochschild cohomology of the quotient orbifold of a symplectic variety by symplectic automorphisms.  Finally, in Section~\ref{sect:H_orb} we establish some linear identifications between Hochschild cohomology and orbifold cohomology, and provide a number of conjectures regarding the relationship between the cup product on Hochschild cohomology and the orbifold product on orbifold cohomology.

In Appendix~\ref{sect:cat_coho}, by Pieter Belmans, the Hochschild cohomology of the category $\mrm{Qcoh}(\msc{X})$ is identified with the derived endomorphisms of the pushforward $\Delta_\ast\O_\msc{X}$.  In Appendix~\ref{sect:lin_alg} we cover some basic facts regarding linear representations of finite groups.  In Appendix~\ref{sect:pav}, with Pavel Etingof, we show that a certain group cocycle which appears in the Hochschild cohomology of a symplectic quotient orbifolds is canonically bounded.  This corrects a reccurring error in the literature. 

\subsection{Acknowledgements}

Thanks to the participants and administrators of the Hochschild Cohomology in Algebra, Geometry, and Topology Workshop held at the Oberwolfach Research Institute for Mathematics in February 2016, where this project began in earnest.  Thanks especially to Dmitry Kaledin and Damien Calaque for many helpful conversations, and to Pieter Belmans for helpful comments on a draft and pointing out references.  We would like to thank Andrei C\u{a}ld\u{a}raru for many correspondences and useful clarifications, and Wendy Lowen and Michel Van den Bergh for explaining some aspects of categorical Hochschild cohomology.  Thanks also to Heather Macbeth for offering many useful insights regarding symplectic geometry, to Allen Knutson for writing programs relevant to Remark \ref{r:other-hom}, to Haiping Yang for comments on this remark, and to Lie Fu for helpful comments on the Cohomological Hyperk\"{a}hler Resolution Conjecture. 
\par

Thanks to the Hausdorff Institute for Mathematics in Bonn, Germany, where this project was continued during the Fall 2017 Junior Trimester Program on Symplectic Geometry and Representation Theory. The second author would also like to thank the Max Planck Institute for Mathematics for its hospitality on two research visits.

\section{Remarks on degeneration}
\label{sect:rem_deg}

Before we begin our main endeavor of computing Hochschild cohomology, let us make a few specific remarks on degeneration of the spectral sequence~\eqref{eq:105} and subsequent identification of the associated graded algebra of Hochschild cohomology.
%We fix $X$ a smooth quasi-projective variety equipped with an action of a finite group $G$, and take $\msc{X}=[X/G]$.

\subsection{Examples of degeneration}

As mentioned in the introduction, for affine $X$, derived global sections vanish and we have $\HH^\bullet(\msc{X})=\Gamma(X,\oplus_{g\in G}\!\ T^{poly}_{X^g}\otimes_{\O_X} \det(N_{X^g}))^G$.  Of course, such vanishing holds in many non-affine examples:

\begin{corollary}\label{cor-van}
Suppose that $G$ is a finite group acting on a smooth quasi-projective variety $X$ such that $H^i(X^g, 
%(\wedge^j T_X)|_{X^g}
\wedge^j T_X \otimes \det(N_{X^g})
) = 0$ for $i > 0$ and all $j \geq 0$ and $g \in G$. Then there is an algebra isomorphism \eqref{eq:151-nogr}.
For instance, these conditions hold when $X=\mathbb{P}^n$.
\end{corollary}

The first assertion of this corollary is immediate; we sketch the
proof of the second assertion.  We can assume that $G$ acts faithfully
on $X=\mathbb{P}^n$, so that $G < \operatorname{PGL}_{n+1}$.  Every
element $g \in G$ has finite order, so is represented by a semisimple
element of $\operatorname{GL}_{n+1}$. Then $X^g$ is a disjoint union
of the projectivizations of the eigenspaces of $g$ in $k^{n+1}$.
Therefore $X^g$ is smooth and each connected component is
$\mathbb{P}^m \subseteq \mathbb{P}^n$. There is an isomorphism
$T_{\mathbb{P}^m} \oplus \mathscr{O}_{\mathbb{P}^m}(1)^{n-m} \to
T_{\mathbb{P}^n}|_{\mathbb{P}^m}$,
where the first component maps via the inclusion, and the second maps
by restricting the Euler surjection
$\msc{O}_{\mathbb{P}^n}(1)^{n+1} \twoheadrightarrow
T_{\mathbb{P}^n}$
to the last $n-m$ components and then restricting from $\mathbb{P}^n$
to $\mathbb{P}^m$. Thus $\det(N_{X^g}) = \mathscr{O}_{\mathbb{P}^m}(n-m)$. We therefore
have to show that all higher cohomology of $\wedge^j T_{\mathbb{P}^m}(n-m)$ vanishes.
%$\wedge^j T_{\mathbb{P}^n}|_{\mathbb{P}^m}$ is
%a direct sum of bundles of the form
%$\wedge^{j'} T_{\mathbb{P}^m}(j-j')$ for $j' \leq j$.  It is
%well-known that these bundles all have vanishing higher cohomology 
This is well-known (see, e.g., \cite{Bot-hvb}).
We give an elementary proof using the Euler exact sequence
$0 \to \mathscr{O}_{\mathbb{P}^m} \to
\mathscr{O}_{\mathbb{P}^m}(1)^{m+1} \to T_{\mathbb{P}^m} \to 0$:
Replace $T_{\mathbb{P}^m}$ by the complex
$\mathscr{O}_{\mathbb{P}^m} \to \mathscr{O}_{\mathbb{P}^m}(1)^{m+1}$,
placing in degrees $-1$ and $0$, and take its $j$-th exterior
power. We get a complex in nonpositive degrees quasi-isomorphic to
$\wedge^{j} T_{\mathbb{P}^m}$.  All terms of this complex are sums of
line bundles $\mathscr{O}_{\mathbb{P}^m}(\ell)$ for $\ell \geq 0$.
 The same is true after twisting by
$\mathscr{O}_{\mathbb{P}^m}(n-m)$ with $n-m \geq 0$.  Therefore taking derived global sections of the
complex results in a complex of vector spaces concentrated in
nonpositive degrees. So
$H^i(\mathbb{P}^m, \wedge^{j} T_{\mathbb{P}^m}(n-m))=0$ for $i \geq 1$.

\begin{remark}\label{r:other-hom}
It seems likely that the conditions of the corollary hold for some other  homogeneous spaces.  Let $G$ be a connected reductive group, $P<G$ a parabolic subgroup, and $X=G/P$. Then any finite subgroup $\Gamma < G$ acts by automorphisms on $G/P$. 
%% Note: in many cases the condition $\Gamma < G$ is vacuous, but not
%% always: see
%% \url{https://mathoverflow.net/questions/160292/automorphism-group-of-flag-manifolds}
Now let $g \in G$. In the case $g = 1 \in \Gamma$, so $X^g=X$, then
the relevant cohomology $H^i(X, \wedge^j T_X)$ is described in
\cite{Bot-hvb} (although making it explicit runs into the difficulty
that $\wedge^j T_X$ is not, in general, induced by a semisimple
representation of $P$; see also \cite[Exercise 4, Chapter 4]{Wey-cvbs}
for the vanishing for $i > 0$ in the Grassmannian case).  Next consider an element
$g \in \Gamma$.  Since it has finite order, it is semisimple, so its
centralizer $L:=Z_G(g)$ is reductive and contains a maximal torus $T$
of $G$. Replacing $P$ by a conjugate parabolic we may assume that
$T < P$. We have
%  whose centralizer is a Levi subgroup $L < G$
% (this condition is vacuous when $G=\operatorname{GL}_n$).
% %% But not vacuous in general, see, e.g.,
% %% \url{https://mathoverflow.net/questions/118389/center-of-the-centralizer-of-semisimple-element}
$X^g = \{hP \in X \mid ghP=hP\}$. 
%$LP/P \cong L / (L \cap P)$. 
The connected components of this are all isomorphic to $L^\circ/(L^\circ \cap P)$, where $L^\circ$ is the connected component of the identity of $L$. This is itself a homogeneous space of the same form, since $L^\circ \cap P$ is a parabolic subgroup of $L^\circ$.
Thus, one can again try to apply \cite{Bot-hvb} to each of these spaces.

We note, however, that the higher cohomology appearing in Corollary
\ref{cor-van} does not vanish for all homogeneous spaces (possibly not
even for most).  Indeed, already for $g=1$ and $X$ the complete flag
variety ($G=\operatorname{GL}_n$ and $P=B$ a Borel subgroup), the
higher cohomology fails to vanish in general.  For example, A.~Knutson
provided Macaulay2 code which demonstrates that for $n=6$, the Euler
characteristic of an isotypic component of $H^*(\wedge^4 T_{G/B})$
under the action of G is negative, and hence there exists a
nonvanishing cohomology group in odd degree.  We have been informed
that P.~Belmans extended this using SageMath to flag varieties for
several other groups: in type $D_5$,
$H^{\text{odd}}(\wedge^k T_{G/B})$ does not vanish for
$k \in \{4,5,6\}$, and in type $F_4$,
$H^{\text{odd}}(\wedge^k T_{G/B})$ does not vanish for
$k \in \{3,4,5,6\}$. Moreover, ongoing work of P.~Belmans and
M.~Smirnov is investigating the case where $P$ is a maximal parabolic.
%In particular the case $G=\operatorname{GL}_n$ is tractable, and it would be interesting to determine if the conditions of the corollary apply.
\end{remark}

\subsection{$\gr\HH^\bullet(\msc{X})$ versus $\HH^\bullet(\msc{X})$ in general}

For general $\msc{X}$ there does not exist an algebra isomorphism $\HH^\bullet(\msc{X})\cong \gr\HH^\bullet(\msc{X})$ (as in the above projective examples).  Indeed, in our proof of the non-formality of Hochschild cohomology for projective Calabi--Yau quotient orbifolds (Theorem~\ref{thm:informal}) we observe an explicit obstruction to the existence of such an algebra isomorphism.  So, for example, if we consider a projective symplectic quotient $\msc{X}$ with singular coarse space $X/G$ then we have explicitly $\HH^\bullet(\msc{X})\ncong \gr\HH^\bullet(\msc{X})$.  We refer the reader to Section~\ref{sect:formality}, and in particular the proof of Theorem~\ref{thm:informal}, for more details.
\par

We note that in the case of a projective symplectic quotient $\msc{X}$ we expect that the Hochschild cohomology $\HH^\bullet(\msc{X})$ is strongly related to the orbifold cohomology of $\msc{X}$.  The associated graded ring $\gr\HH^\bullet(\msc{X})$ in this case is expected to be the associated graded ring with respect to a natural ``codimension" filtration on orbifold cohomology.  We discuss this topic in depth in Section~\ref{sect:H_orb}.

\section{Preliminaries}
\label{sect:bg}

%Throughout $k$ is a field either of characteristic $0$
%or of a characteristic larger
%than the dimension of $X$ which does not divide the order of the finite
%group $G$.
%(except, for the degeneration result, Section
%\ref{ss:form-deg}, the characteristic must also fit the needs of
%\cite{arinkinetal14}).
Here we fix notations and record some basic
and well-known facts about fixed loci and resolutions which we will use throughout.

\subsection{Notations}
\label{sect:global_notes}

A {\it variety} is a quasi-projective scheme over $\operatorname{Spec} k$.  For any variety $X$ we let $\O(X)$
denote the algebra of global functions.  For
any variety map $\varphi:X\to Y$ we let
$\bar{\varphi}:\O_Y\to \varphi_\ast\O_X$ denote the underlying map of
sheaves of rings. By a vector bundle we will always mean a locally
free sheaf of {\it finite rank}.
\par

For a subvariety $i:Z\subset X$ and a quasi-coherent sheaf $M$ on $X$ we always let
$M|_Z$ denote the pullback $i^\ast M$ on $Z$, as opposed to $i^{-1}M$.  At a geometric point $p:\operatorname{Spec}\bar{k}\to X$ we let $\O_{X,p}$ denote the corresponding local ring of the base change $X_{\bar{k}}$.
\par

Suppose $X$ is a variety with the action of a finite group $G$, which
we always assume to be an action by automorphisms of $X$ over
$\operatorname{Spec} k$.  By applying the contravariant functor $\Gamma(X,?)$, the
action map $G\to \Aut_{\mathrm{sch}}(X)$ produces a map
$G^{op}\to \Aut_{\mathrm{alg}}(\O(X))$.  Whence we get an action of
the opposite group $G^{op}$ on the global sections of $\O_X$, and an
action of $G$ by precomposing with the inversion operation on $G$.
Similarly, for any $G$-stable open $U$ we will get an action of $G$ on
$\Gamma(U,\O_X)$.  For a group $G$ and $G$-representation $V$ in
general we often take ${^gv}=g\cdot v$, for $g\in G$ and $v\in V$.

Throughout $X$ will be a smooth (quasi-projective) variety equipped with
the action of a finite group $G$.  We assume that the characteristic of
the base field $k$ does not divide $G$ and is greater than the
dimension of $X$.
% $G$, and in fact all finite groups, have order coprime to the
% characteristic o%f the base field $k$, and moreover we .
We will often use the fact that such an $X$ admits a covering by
affines $\{U_i\}_i$ such that each $U_i$ is preserved under the action
of $G$.  See~\cite[Appendix A]{mustata} or~\cite{grothendieck57}.
\par

To ease notation we omit the bullet from Ext and Hom notation, e.g.\ $\sExt_Y(M,N):=\sExt^\bullet_Y(M,N)$.

\subsection{Fixed locus under group actions}

We recall that a given variety $Y$ is smooth (over $k$) if and only if for any open affine subset $U\subset Y$ and any closed embedding $U\to \mathbb{A}^n_k$ with ideal of definition $\mathscr{I}$, the sequence
\[
0\to \mathscr{I}/\mathscr{I}^2\to \Omega_{\mbb{A}^n}|_U\to \Omega_U\to 0
\]
is split exact.  A closed subvariety $Z\to Y$ in a smooth variety $Y$ is smooth if and only if the analogous sequence
\[
0\to \mathscr{I}_Z/\mathscr{I}_Z^2\to \Omega_Y|_Z\to \Omega_Z\to 0
\]
is exact with $\Omega_Z$ locally free.
\par

%As in Section~\ref{sect:RG}, for
For any $g\in G$ we take $X^g$ to be the subvariety in $X$ fixed by $g$.  This is the preimage of the graph of $g$ along $\Delta$, $X^g=\Delta^{-1}(\mathrm{graph}(g))$.  Alternatively, it is the fiber product of $X$ with itself over $X^2$ along the two maps $\Delta:X\to X^2$ and $(1\times g)\Delta:X\to X^2$.  For an arbitrary subset $S\subset G$ we let $X^S$ denote the intersection $\cap_{g\in S} X^g$ in $X$.

\begin{lemma}\label{lem:gbasics}
\begin{enumerate}
\item[(i)] Suppose $U$ is any open in $X$ with $g(U)=U$ for some $g\in G$.  Then $U^g=X^g\cap U$.
\item[(ii)] Any point $p$ fixed by a subgroup $G'\subset G$ has an affine neighborhood $U$ with $g(U)=U$ for all $g\in G'$.
\item[(iii)] When $U$ is affine we have $\O(U^g)=\O(U)/(f-{^gf})_{f\in \O(U)}$.
\item[(iv)] The points of $X^g$ are all the points $p:\Spec k\to X$ with $g(p)=p$.
\end{enumerate}
\end{lemma}

\begin{proof}
(i) One simply shows that $X^g\cap U$, along with its inclusion into $U$, has the universal property of the fiber product.  (ii) For an arbitrary affine neighborhood $W$ of $p$ we can take $U=\cap_{g\in G'}g(W)$.  (iii) First note that $U^g$ will be affine, so we need only check that the map $\O(U)\to \O(U)/(f-{^gf})_{f\in \O(U)}$ has the appropriate universal property among maps to rings.
\par

For any algebra map $\phi:\O(U)\to B$ such that the two morphisms
\begin{equation}\label{eq:115}
\O(U)\ot \O(U)\overset{1\ot 1}{\underset{1\ot \bar{g}}\rightrightarrows} \O(U)\ot \O(U)\overset{\mathrm{mult}}\to \O(U)\overset{\phi}\to B
\end{equation}
agree we will have $\phi(a(f-\bar{g}(f)))=0$ for all $a,f\in \O(U)$.  Replace $f$ with $-{^gf}=-\bar{g}^{-1}(f)$ to see that $\phi$ factors uniquely through $\O(U)/(f-{^gf})$.  Conversely, for any map $\phi:\O(U)\to B$ factoring through the quotient we will have that the two maps~\eqref{eq:115} agree.  So $\O(U)/(f-{^gf})$ has the appropriate universal property, $U^g=\Spec(\O(U)/(f-{^gf}))$, and $\O(U^g)=\O(U)/(f-{^gf})$.  Claim (iv) follows by definition
(alternatively, one can reduce
to the affine case by (ii) and (i) and apply (iii)).
\end{proof}

\begin{lemma}\label{lem:smoothint}
For any subset $S\subset G$ the fixed locus $X^S$ is smooth.  In particular, at any geometric point $p$ of the fixed space $X^S$ there is an isomorphism between the completed symmetric algebra $\hat{\mrm{Sym}}((T_p^\ast)_{S})$ generated by the $\langle S\rangle$-coinvariants of the tangent space at $p$ and the complete local ring $\hat{\O}_{X^S,p}$.
\end{lemma}

\begin{proof}
By changing base we may assume $k$ is algebraically closed.  Note that $X^S=X^H$, where $H$ is the subgroup of $G$ generated by $S$.  So we may also assume $S=G$.  For smoothness, it suffices to show now that the complete local rings $\hat{\O}_{X^{G},p}$ at closed points $p$ on the fixed locus are power series rings~\cite{matsumura}.
\par

We may assume $X$ is affine.  For any such point $p$, with corresponding maximal ideal $\mfm_p\subset \O(X)$, we have the exact sequence $0\to \mfm_p^2\to \mfm_p\to \mfm_p/\mfm_p^2\to 0$ which admits a splitting over $G$, as $k[G]$ is semisimple.  Any such splitting $T^\ast_p=\mfm_p/\mfm_p^2\to \mfm_p$ gives compatible $G$-module algebra isomorphisms $\mathrm{Sym}(T^\ast_p)/(T^\ast_p)^n\to \O_{X,p}/\mfm_p^n$ for each $n>0$, where $G$ acts linearly on the domain.
\par

Take $R_n=\mathrm{Sym}(T^\ast_p)/(T^\ast_p)^n$ and let $I=(f-{^gf})_{f\in \O(X),g\in G}$ to be the ideal of definition for $X^G$.  Let $\bar{\mfm}_p$ be the maximal ideal for the point $p$ in $X^G$.  Then we have
\[
\frac{\mathrm{Sym}((T^\ast_p)_G)}{((T^\ast_p)_G)^n}=\frac{R_n}{(v-{^gv})_{v\in T^\ast_p,g}}=\frac{R_n}{(f-{^gf})_{f\in R,g}}\overset{\cong}\to \frac{\O(X)}{(I+\mfm_p^n)}=\frac{\O(X^G)}{\bar{\mfm}_p^n}
\]
and 
\[
\hat{\mathrm{Sym}}((T^\ast_p)_G)\cong \varprojlim_n\O(X^G)/\bar{\mfm}_p^n=\hat{\O}_{X^G,p}.
\]
Thus $X^G$ is smooth, and the complete local rings are as described.
\end{proof}

\subsection{Resolution property for stacks and the diagonal}

For a stack $\msc{Y}$ we say $\msc{Y}$ has the resolution property if any coherent sheaf $M$ on $\msc{Y}$ admits a surjection $E^0\to M$ from a vector bundle on $\msc{Y}$.  The following lemma is well-known.  One can see for example the comments of the introduction to~\cite{thomason87}.

\begin{lemma}\label{lem:resolprop}
Let $Y$ be a quasi-projective variety equipped with the action of a finite group $H$.  Then the quotient stack $[Y/H]$ has the resolution property. 
\end{lemma}

\begin{proof}
As in the proof of~\cite[Proposition IV.1.5]{knutson}, one can show that $Y$ admits a locally closed $H$-equivariant embedding $Y\to \mathbb{P}^n$ where $H$ acts linearly on projective space.  Linearity of the $H$-action implies that the twist $\O(1)$ is an equivariant ample line bundle on $\mathbb{P}^n$, and the pullback $\O_Y(1)$ is therefore an equivariant ample line bundle on $Y$.  Then by a slight variant of the usual argument (see e.g.~\cite[Corollary 5.18]{hartshorne}) one concludes that any equivariant coherent sheaf $M$ on $Y$ admits a equivariant surjection $E^0\to M$ from an equivariant vector bundle $E^0$ on $Y$.
\end{proof}

The resolution property allows us to derive the sheaf hom functor $\sHom_{[Y/H]}$ via resolutions by vector bundles.  Furthermore, when $Y$ is smooth, as will be the case for us, these resolutions can be taken to be bounded and the derived functor $\mrm{R}\sHom_{[Y/H]}$ takes values in the derived category of coherent sheaves on $[Y/H]$.  We will be applying this in the case $Y=X\times X$ and $H=G\times G$, so that $[Y/H]\cong \msc{X}\times\msc{X}$.  We also include the following well-known result for the convenience of the reader.

\begin{lemma}\label{lem:aff_diag}
Let $\msc{Y}=[Y/H]$ be the quotient stack of a quasi-projective variety by a finite group action.  Then the diagonal $\Delta:\msc{Y}\to \msc{Y}\times\msc{Y}$ is finite, affine, and separated.
\end{lemma}

\begin{proof}
By~\cite[\href{https://stacks.math.columbia.edu/tag/04XD}{tag 04XD}]{stacks-project} it suffices to show that the second projection $\msc{Y}\times_{\msc{Y}\times\msc{Y}}(Y\times Y)\to Y\times Y$ is finite, affine, and separated.  But one calculates directly that the pullback $\msc{Y}\times_{\msc{Y}\times\msc{Y}}(Y\times Y)$ is isomorphic to the product $Y\times H$, and that the projection $Y\times H\to Y\times Y$ in this case is the coproduct of the graphs $(1\times h)\Delta:Y\to Y\times Y$.  This map is indeed finite, affine, and separated.
\end{proof}

\section{The spectral sequence relating $\uHH^\bullet(\msc{X})$ to $\HH^\bullet(\msc{X})$} \label{sect:ss}

In this section we introduce the smash product sheaf $\Delta_\ast\O_X\rtimes G$ on $X^2$ and an equivariant dg algebra $\msc{E}(\msc{X})$ (defined up to quasi-isomorphism) for which the cohomology $H^\bullet(\msc{E}(\msc{X}))$ is the sheaf valued Hochschild cohomology $\uHH^\bullet(\msc{X})$.  We apply results of~\cite{arinkinetal14} to find that $\msc{E}(\msc{X})$ is formal, as a sheaf, and subsequently deduce degeneration of the local-to-global spectral sequence~\eqref{eq:105}.

\subsection{The local algebras $\msc{E}(\msc{X})$ and $\uHH^\bullet(\msc{X})$ on $X$}
\label{sect:loc}

Using descent, the local equivariant Hochschild cohomology and its relations to the Hochschild cohomology of $\msc{X}$ are explained as follows. Observe first that sheaves $\mathscr{F}$ on $\X$ are the same as $G$-equivariant sheaves $p^* \mathscr{F}$ on $X$ via the \'etale covering map $p:X \to \X$. We have the following cartesian diagram:
\begin{equation}
\xymatrix{
X \times G \ar[d] \ar[r]^{\widetilde \Delta} & X \times X \ar[d] \\
\X \ar[r]^{\Delta} & \X \times \X,
}
\end{equation}
where $\widetilde \Delta(x,g) = (x,gx)$, i.e., $\Delta_g := \widetilde \Delta|_{X \times \{g\}}$ is the graph of $g: X \to X$.
%{\color{blue}where $\widetilde{\Delta}$ is such that the restriction to the $g$-th copy of $X$ in $X\times G$ is the graph of $g$, i.e. the map $X\to X\times X$ which recovers the identity on the first factor and the action of $g$ on the second factor.}  
Therefore, applying base change, we have
\begin{equation}
  (p^* \times p^*) \sREnd_{\X^2}(\Delta_*\O_{\X}) = 
\sREnd_{X^2}(\widetilde \Delta_* (\O_X \otimes \O_G)).
\end{equation}
We can identify $\widetilde \Delta_*(\O_X \otimes \O_G)$ with
the sheaf  $\Delta_* \O_X \rtimes G$ (with $\O_{X^2}$ action given by the 
usual $\O_X$-bimodule structure).  

 We can therefore study the Hochschild cohomology of $\X$ via the
 $G \times G$-equivariant algebra $\sREnd_{X^2}(\Delta_* \O_X \rtimes G)$
 in the derived category of quasi-coherent sheaves on $X^2$. In particular, $\HH^\bullet(\X) = \mrm{R}\Gamma(\X^2, \sREnd_{\X^2}(\O_{\X})) = \mrm{R}\Gamma(X^2, \sREnd_{X^2}(\Delta_* \O_X \rtimes G))^{G \times G}
= \Ext_{X^2}(\Delta_* \O_X \rtimes G,\Delta_* \O_X \rtimes G)^{G \times G}$. Note that we can
rewrite this as $\Ext_{X^2}(\Delta_* \O_X, \Delta_* \O_X \rtimes G)^G$, where the $G$ acts via the diagonal action.
%We therefore get a multiplicative spectral sequence from $\Ext_{X^2}(\Delta_* O_X \rtimes G, \Delta_* \O_X \rtimes G)^{G \times G}$ to $\HH^\bullet(\X)$.

On the derived level, we also have
%To simplify this and  connect back to our original expression, 
%observe that 
$\sREnd_{X^2}(\Delta_* \O_X \rtimes G) \cong \sRHom_{X^2}(\Delta_* \O_X, \Delta_* \O_X \rtimes G) \rtimes G$ in the derived category of quasi-coherent sheaves on $X^2$. Pushing forward under
the projection $\widetilde \pi: X^2 \to X \times X/G$
%${\color{blue}\pi}: X^2 \to X \times X/G$ and
 and taking invariants under $\{1\} \times G$ yields the $G$-equivariant algebra \begin{equation}
\msc{E}(\msc{X}) := \widetilde{\pi}_* \sREnd_{X^2}(\Delta_* \O_X \rtimes G)^G \cong 
\widetilde{\pi}_* \sRHom_{X^2}(\Delta_* \O_X, \Delta_* \O_X \rtimes G)
\end{equation}
in the derived category of quasi-coherent sheaves on $X \times X/G$. Note that
the map $X \to X/G$ is finite, and hence the pushforward $\widetilde{\pi}_*$ is underived.

Observe that the cohomology sheaves $H^\bullet (\msc{E}(\msc{X}))$ are scheme-theoretically supported on the image of the diagonal map $\overline{\Delta}: X \to X \times X/G$.  Moreover their restriction to the diagonal is precisely $\uHH^\bullet(\msc{X})$. We obtain the multiplicative spectral sequence \eqref{eq:105}.

Explicitly, the multiplication on $\msc{E}(\msc{X})$ can be
described as the sum over $g,h \in G$ of the bilinear maps
\begin{multline}\label{e:exp-mult}
  \widetilde{\pi}_* \sRHom_{X^2}(\Delta_* \mathscr{O}_X, (\Delta_g)_*\mathscr{O}_X ) \times  \widetilde{\pi}_* \sRHom_{X^2}(\Delta_* \mathscr{O}_X, (\Delta_h)_*\mathscr{O}_X ) \\ \to \widetilde{\pi}_* \sRHom_{X^2}(\Delta_* \mathscr{O}_X, \Delta_*(\Delta_{gh})_* \mathscr{O}_X )
\end{multline}
given by the identification
$\widetilde{\pi}_* \sRHom_{X^2}(\Delta_* \mathscr{O}_X, \Delta_*\mathscr{O}_X
\rtimes g) \cong \widetilde{\pi}_* \sRHom_{X^2}(\Delta_* \mathscr{O}_X \rtimes h,
\Delta_*\mathscr{O}_X \rtimes gh)$
from the diagonal action of $g$ followed by composition.  Taking
cohomology sheaves and restricting to the diagonal yields the same statements for
$\uHH^\bullet(\msc{X})$.

\begin{remark}
\label{rem:intrinsic}

We remark that the dg algebra $\msc{E}(\msc{X})$ is intrinsic to $\msc{X}$, as is the cohomology $\uHH^\bullet(\msc{X})$.  As noted above, sheaves on $\msc{X}\times X/G$ are equivalent to equivariant sheaves on $X\times X/G$, under the $G$-action on the first factor.  Hence $\msc{E}(\msc{X})$ can be viewed as a sheaf of algebras on $\msc{X}\times X/G$.  Indeed, $\msc{E}(\msc{X})$ is the pushforward of the derived endomorphism ring $\mrm{R}\sEnd_{\msc{X}^2}(\Delta_\ast\O_\msc{X})$ along the finite map $\msc{X}^2\to \msc{X}\times X/G$.  (Pushing forward from $\msc{X}$ to $X/G$ in the second factor automatically takes $\{1\}\times G$-invariants.)  Since $X/G$ is the coarse space of $\msc{X}$, this construction of $\msc{E}(\msc{X})$ does not rely on the \'etale covering from $X$.  Similarly, $\uHH^\bullet(\msc{X})$ can be given as the pushforward of $\sExt_{\msc{X}^2}(\Delta_\ast\O_\msc{X},\Delta_\ast\O_\msc{X})$ to $\msc{X}\times X/G$, or as the pushforward to $\msc{X}$ via the first projection.
\end{remark}

% \begin{remark}
%   We did not need to project all the way to $X$: it is enough to
%   consider the finite map $\tilde \pi: X^2 \to X \times X/G$. Since
%   $\tilde \pi_* \sREnd_{X^2}(\Delta_* \O_X \rtimes G)^G$ is
%   scheme-theoretically supported on the image of the diagonal
%   $i: X \to X \times X/G$, we have
%   $\tilde \pi_* \sREnd_{X^2}(\Delta_* \O_X \rtimes G)^G \cong i_*
%   \msc{E}(\msc{X})$.
% \end{remark}

\subsection{Formality and degeneration}\label{ss:form-deg}
By \cite[Theorem 4.4]{arinkinetal14} (see also Theorem 4.9), the
object in the derived category of quasicoherent complexes on $X^2$,
\begin{equation}\label{eq:dsi}
(\Delta_* \O_X \rtimes G) \otimes^L_{\O_{X^2}} (\Delta_* \O_X \rtimes G) 
\end{equation}
is formal. This object is $G \times G$-equivariant for the action of $G \times G$ on $X^2$. The second factor acts freely in the sense that, for $R^{(2)}_g: X \times X \to X \times X$ the action map on this factor, $R^{(2)}_g(x_1,x_2) = (x_1, g \cdot x_2)$, we have
\begin{equation}
(\Delta_* \O_X \rtimes G) \otimes^L_{\O_{X^2}} (\Delta_* \O_X \rtimes G) 
\cong \bigoplus_{g \in G}  (R_g)_* (\Delta_* \O_X \otimes^L_{\O_{X^2}} (\Delta_* \O_X \rtimes G)).
\end{equation}
Thus, it suffices to show that the following object is formal:
% and since the second factor of $G$ acts freely, we equivalently have on invariants the formality of
\begin{equation}
\Delta_* \O_X \otimes^L_{\O_{X^2}} (\Delta_* \O_X \rtimes G).
\end{equation}
%As observed after \cite[Theorem 4.9]{arinkinetal14}, 
We do this using coherent duality.  For every scheme $Y$, let $\pr_Y:  \to \Spec k$ be the
projection to a point. Let $\DD(\mathcal{F}) :=
\sRHom_Y(\mathcal{F}, \pr_Y^! k)$. If $Y$ is smooth, then the (shifted) dualizing
complex $\pr_Y^! k$ equals $\omega_Y[\dim Y]$,
where $\omega_Y$ is the canonical line bundle.  For $f: Y \to Z$ a proper
map, Grothendieck--Serre duality provides
a quasi-isomorphism  $Rf_* \sRHom_Y(\mathcal{F}, f^! \mathcal{G}) \cong \sRHom_Z (Rf_* \mathcal{F}, \mathcal{G})$. We then have
\[
Rf_* \DD \mathcal{F} := Rf_* \sRHom(\mathcal{F}, f^! \pr^! k)
\cong  \sRHom(Rf_* \mathcal{F}, \pr^! k) =: \DD(Rf_* \mathcal{F}).
\]
Thus $Rf_*$ preserves the class of self-dual objects under $\DD$.

Now let us specialize to the case $Y=X, Z=X^2$.
Let  $\Delta_g: X \to X \times X$,
the graph of $g: X \to X$. Then $\Delta_* \O_X \rtimes G = \bigoplus_{g \in G}
(\Delta_g)_* \O_X$, with each term self-dual under $\DD$ (note that $R\Delta_* = \Delta_*$ since $\Delta$ is an embedding). Therefore,
\begin{equation}\label{e:selfdual-doxg}
\DD(\Delta_* \O_X \rtimes G) \cong \Delta_* \O_X \rtimes G.
\end{equation}

Using the preceding identity in the second isomorphism, we obtain:
\begin{equation}
(\Delta_* \O_X \otimes^L_{\O_{X^2}} \DD(\Delta_* \O_X \rtimes G))
\cong \sRHom_{X^2}(\Delta_* \O_X, \DD (\Delta_* \O_X \rtimes G)) \cong
\sRHom_{X^2}(\Delta_* \O_X, \Delta_* \O_X \rtimes G),
\end{equation}
with the first isomorphism following from adjunction,
$\sRHom(\mathscr{E} \otimes \mathscr{F}, \mathscr{G}) \cong
\sRHom(\mathscr{E}, \sRHom(\mathscr{F}, \mathscr{G}))$.

%% So $\Delta_* \O_X \rtimes G)$
%%s self-dual, as it is the sum of pushforwards of sel
%%If $Y$ is smooth and $\mathcal{F}$ is a vector bundle, this is the vector bundle dual
%%%and $Y$ 
% \begin{multline}\label{eq:dvb}
% (i_* \mathcal{F})^\vee := 
% \sRHom_{Z}(i_* \mathcal{F}, \omega_Z[\dim Z]) \cong
% i_*\sRHom_Z(\mathcal{F}, i^! \omega_Z[\dim Z]) \\ \cong
% i_* (\mathcal{F}^\vee \otimes^L i^! \omega_Z[\dim Z])
% \cong i_*(\mathcal{F}^\vee \otimes^L (Li^*(\omega_Z[\dim Z]^\vee))^\vee)
% \cong i_*(\mathcal{F}^\vee).
% \end{multline}
% \begin{multline}\label{eq:dodual}
%   (\Delta_* \O_X)^\vee := \sRHom_{X^2}(L\Delta_* \O_X, \omega_{X^2}[\dim X^2]) \cong \sRHom_X(\O_X, \Delta^! \omega_{X^2}[\dim X^2])
%   \\ \cong \Delta^! \omega_{X^2}[\dim X^2] \cong (L\Delta^* (\omega_{X^2}[\dim X^2])^\vee)^{\vee}
%   \cong (L\Delta^* \O_{X^2})^{\vee} \cong \Delta_* \O_X.
% \end{multline}
%using the subscript of $X$ to denote duality on $X$.
% Taking $\mathcal{F} = \O_X \rtimes G$ and $i=\Delta$,
% it follows that $\Delta_* \O_X \rtimes G$ is self-dual: $(\Delta_* \O_X \rtimes
% G)^\vee \cong (\Delta_* \O_X \rtimes G)$.
% either applying the same argument for
%each summand $\Delta_* \O_X \cdot g$, or applying the action of $G$.

Next, the cohomology sheaves of \eqref{eq:dsi}, which is just the sum
of derived intersections of graphs of elements $g \in G$, consists of
vector bundles on the fixed loci $X^g$, which are smooth closed
subvarieties of $X$. This is well-known (see, e.g., \cite[Proposition
A.5]{caldararuetal03}) and is also part of the statement of
\cite[Theorem 4.4]{arinkinetal14}.  Then as above, the
 dual of each cohomology sheaf is also a sum of vector bundles, the duals
to the ones before.  Therefore
% For any vector bundle
% $\mathscr{F}$ on $X^g$, for $i_g: X^g \subseteq X$ the inclusion, then
% the same computation as in \eqref{eq:dodual} shows that
% $((i_g)_* \mathscr{F})^\vee \cong (i_g)_* (\mathscr{F}^\vee)$. Hence
formality of \eqref{eq:dsi} implies formality of its dual,
$\sRHom_{X^2}(\Delta_* \O_X, \Delta_* \O_X \rtimes G)$.

% Next, by \eqref{eq:123} the cohomology sheaves $H^\bullet
% \msc{E}(\msc{X}) =: \uHH^\bullet(\O_X, \O_X \rtimes G)$ is a direct
% sum of vector bundles on fixed-point loci $X^g \subseteq X$, which are
% smooth closed subvarieties of $X$ (the proof of \eqref{eq:123} is
% given in Section \ref{sect:polyveder} without using any results of
% this section, and furthermore we only need the statement to hold
% without the algebraic structure, where it is known and also a
% consequence of \cite[Theorem 4.4]{arinkinetal14}).

% From this we conclude that
% $\sRHom_{X^2}(\Delta_* \O_X, \Delta_* \O_X \rtimes G)$ is formal in
% the category of complexes of quasi-coherent sheaves on $X^2$.
% Since $\sRHom_{X^2}(\Delta_* \O_X, \Delta_* \O_X \rtimes G)$ is
% supported on the union of the graphs of $g: X \to X$, which is finite
% over $X$, it follows that its $G$-invariant pushforward
% $\msc{E}(\msc{X})$ to $X$ is also a formal complex.  
Since the cohomology sheaves of $\msc{E}(\msc{X})$ are canonically isomorphic to $\overline{\Delta}_* \uHH^\bullet(\msc{X})$ (for $\overline{\Delta}: X \to X \times X /G$ the diagonal map), we conclude that the spectral sequence \eqref{eq:105} is quasi-isomorphic to the
one computing the hypercohomology of $\overline{\Delta}_* \uHH^\bullet(\msc{X})$. The quasi-isomorphism induces an isomorphism of bigraded vector spaces on the $E_2$-page, and for the hypercohomology of $\overline{\Delta}_* \uHH^\bullet(\msc{X})$, we have $E_2=E_\infty$. Thus the same must be true for the spectral sequence computing the hypercohomology of $\msc{E}(\msc{X})$.

\section{Necessary relations between normal bundles and tangent sheaves}
\label{sect:SPprelim}

We cover some relations between the normal bundles for the fixed spaces $X^g$ and the restrictions of the tangent sheaf $T_X|_{X^g}$.  These relations are required for our description of the cup product on the cohomology $\uHH^\bullet(\msc{X})$.  In particular, these relations ensure that our geometric description of the cup product is well-defined.

\subsection{Decompositions of the (co)tangent sheaf}
\label{sect:decomp}

For each $g\in G$ we consider the restriction $T_X|_{X^g}$.  Note that $g$ now acts by sheaf automorphisms on $T_X|_{X^g}$.  Since $X^g$ is a smooth embedded subvariety in $X$, we have the exact sequence
\[
0\to T_{X^g}\to T_X|_{X^g}\to N_{X^g}\to 0.
\]
We claim that this sequence identifies $T_{X^g}$ with the $g$-invariants in $T_X|_{X^g}$, i.e. the kernel of the endomorphism $(1-g):T_X|_{X^g}\to T_X|_{X^g}$.  Hence $N_{X^g}$ is canonically identified with the image $(1-g)T_X|_{X^g}$.  This fact can readily be checked complete locally, as we do below.  It follows that we have canonical decompositions
\[
T_X|_{X^g}\cong T_{X^g}\oplus N_{X^g}\ \ \mrm{and}\ \ T^{poly}_X|_{X^g}\cong T^{poly}_{X^g}\ot_{\O_{X^g}}(\wedge_{\O_{X^g}}^\bullet N_{X^g}).
\]
In particular, we have a canonical sheaf projection
\[
p_g:T^{poly}_X|_{X^g}\to T^{poly}_{X^g}\ot_{\O_{X^g}}\det(N_{X^g}).
\]
(Cf.~\cite{anno} and~\cite[Sect.~1.4]{negronwitherspoon17}.)  There is a similar decomposition for the cotangent sheaf $\Omega_X|_{X^g}\cong \Omega_{X^g}\oplus N^\vee_{X^g}$.
\par

Since we would like to make similar complete local reductions throughout the remainder of this document, let us provide an argument for the above claim in detail.  The reader will not be harmed by skipping this on a first reading
and proceeding immediately to Section~\ref{sect:codimension}.
\par

We have the exact sequence $0\to N^\vee_{X^g}\to \Omega_X|_{X^g}\to \Omega_{X^g}\to 0$, which is locally split because $\Omega_{X^g}$ is locally free.  Since $g$ has finite order, $\Omega_X|_{X^g}$ splits into eigensheaves $\Omega_1\oplus \Omega_1^\perp$, where $\Omega_1$ is the trivial eigensheaf for the action of $g$ and $\Omega_1^\perp$ is the sum of all non-trivial eigensheaves.  Since the projection $\Omega_X|_{X^g}\to \Omega_{X^g}$ is $g$-linear, we have that $\Omega_1^\perp$ is in the kernel.  Hence we have the proposed splitting $\Omega_X|_{X^g}\cong \Omega_{X^g}\oplus N^\vee_{X^g}$ if and only if $\Omega_{X^g}$ and $\Omega_1$ have the same rank.  This property can be checked locally, or complete locally.
\par

At a geometric point $p$ of $X^g$ we take a $g$-linear section $T_p^\ast\to m_p$, where $m_p\subset \O_{X,p}$ is the maximal ideal, and let $V$ denote the image of this section.  Note that $V$ is a subrepresentation in $\O_{X,p}$ for the action of the semisimple operator $g$.  We are free to take a basis $\{x_1,\dots, x_n\}$ of $g$-eigenvectors for $V$.
\par

If $k$ is already algebraically closed, the elements $d_i=d_{x_i}$ provide a basis of $g$-eigenvectors for $\Omega_X|_{X^g}$ in an affine neighborhood around $p$.  In general we may work complete locally to obtain a $g$-linear isomorphism
\begin{equation}\label{eq:cotang}
\hat{\O}_{X^g,p}\ot V\overset{\cong}\longrightarrow \Omega_{\hat{X}_p}|_{\hat{X^g}_p},\ \ f\ot x_i\mapsto fd_i.
\end{equation}
Dually, we have a $g$-linear isomorphism
\begin{equation}
\hat{\O}_{X^g,p}\ot V^\ast\overset{\cong}\longrightarrow T_{\hat{X}_p}|_{\hat{X^g}_p},\ \ f\ot x_i^\ast\mapsto fd_i^\vee
\end{equation}
where $d_i^\vee:\Omega_{\hat{X}_p}\to \hat{\O}_{X,p}$ is the dual function to $d_i$.
\par

From the above isomorphisms we see that the invariants $\Omega_1$ and $T_1$ are both of rank $\dim(V^g)$ around $p$.  From Lemma~\ref{lem:smoothint} we see that $\Omega_{X^g}$ is of rank $\dim(V^g)=\dim(X^g)|_p$ around $p$ as well.  Therefore
\[
\mrm{rank}(\Omega_{X^g})=\mrm{rank}(\Omega_1),\ \ \mrm{rank}(\Omega_1^\perp)=\mrm{rank}(N_{X^g}^\vee),\ \ \mrm{rank}(T_{X^g})=\mrm{rank}(T_1),\ \ \mrm{rank}(T_1^\perp)=\mrm{rank}(N_{X^g}).
\]
This implies that the inclusions $T_{X^g}\to T_X|_{X^g}$ is an isomorphism onto the invariants and, since $(1-g)T_X|_{X^g}=T_1^\perp$, the projection $T_X|_{X^g}\to N_{X^g}$ provides an isomorphism from the perpendicular
\begin{equation}\label{eq:793}
(1-g)T_X|_{X^g}\cong N_{X^g}.
\end{equation}
A similar result holds for the cotangent sheaf, and hence we have the proposed splittings.

\subsection{Codimensions and components of the fixed spaces}
\label{sect:codimension}

Recall that the codimension of a smooth subvariety $Y\subset X$ is a section of the sheaf of integers on $Y$, i.e. a continuous function $\mathrm{codim}(Y):Y\to \mathbb{Z}$.  For sections of the sheaf of integers over a topological space we write $s_1<s_2$, $s_1\leq s_2$, etc. if the comparison holds pointwise.  For any subspace $Z\subset Y$ we can restrict to get a section $\mathrm{codim}(Y)|_Z:Z\to Y\to \mathbb{Z}$ of the sheaf of integers over $Z$.  Note that when $Z$ is connected $\codim(Y)|_Z$ is a constant.

\begin{lemma}\label{lem:codimineqs}
For any $g,h\in G$, and component $Y$ of $X^g\cap X^h$, we have
\[
\codim(X^{gh})|_Y\leq \codim(X^g\cap X^h)|_Y\leq \codim(X^g)|_Y+\codim(X^h)|_Y.
\]
If, additionally, $\codim(X^{gh})|_Y= \codim(X^g)|_Y+\codim(X^h)|_Y$ then $Y$ is also a component of $X^{gh}$.
\end{lemma}

\begin{proof}
The first inequality follows simply from the inclusion $X^g\cap X^h\subset X^{gh}$.  The second inequality follows from the fact that we have the surjective map of vector bundles
\[
N_{X^g}^\vee|_{(X^g\cap X^h)}\oplus N_{X^h}^\vee|_{(X^g\cap X^h)}\to N_{X^g\cap X^h}^\vee.
\]
\par

Let $Z$ be a component of $X^{gh}$ containing $Y$.  If we have an equality $\codim(X^{gh})|_Y= \codim(X^g\cap X^h)|_Y$ then, $Y$ is a codimension $0$ smooth embedded subvariety in $Z$.  Since any component of smooth variety is irreducible this implies $Y=Z$.  That is to say $Y$ is also a component of the intersection.  Since we have the implication
\[
\begin{array}{l}
\codim(X^{gh})|_Y= \codim(X^g)|_Y+\codim(X^h)|_Y\\
\Rightarrow \codim(X^{gh})|_Y= \codim(X^g\cap X^h)|_Y
\end{array}
\]
the first inequality also implies that $Y$ is a component of $X^{gh}$.
\end{proof}

Below we describe the vanishing of the cup product on cohomology in terms of the inequality
\[
\codim(X^{gh})|_Y\leq \codim(X^g)|_Y+\codim(X^h)|_Y.
\]

\subsection{Relations between the normal bundles}
\label{sect:(co)normal}

Consider $g,h\in G$ and fix a component $Y$ of the intersection $X^g\cap X^h$.  We then have the canonical map
\begin{equation}\label{eq:kappa}
\kappa:=\kappa(g,h,Y):N_{X^g}|_Y\oplus N_{X^h}|_Y\to N_{X^{gh}}|_Y
\end{equation}
whose restriction to each $N_{X^\tau}|_Y$, for $\tau=g,h$, is the inclusion $N_{X^\tau}|_Y\to T_X|_Y$ determined by the splitting of Section~\ref{sect:decomp} composed with the standard projection $T_{X}|_Y\to N_{X^{gh}}|_Y$.

\begin{lemma}\label{lem:esslem}
Suppose $Y$ is a component of the intersection $X^g\cap X^h$.  The following are equivalent:
\begin{enumerate}
\item[(i)] $Y$ is a shared component with $X^{gh}$ and the intersection of $X^g$ and $X^h$ is transverse along $Y$.
\item[(ii)] $\codim(X^{gh})|_Y=\codim(X^g)|_Y+\codim(X^h)|_Y$.
\item[(iii)] The map $\kappa:N_{X^g}|_Y\oplus N_{X^h}|_Y\to N_{X^{gh}}$ of~\eqref{eq:kappa} is an isomorphism.
\item[(iv)] $N_{X^g}|_Y$ and $N_{X^h}|_Y$ have trivial intersection in $T_X|_Y$.
\end{enumerate}
\end{lemma}

This is a geometric analog of~\cite[Lemma 2.1]{SW}.  We prove~\cite[Lemma 2.1]{SW} in finite characteristic in the appendix, Lemma~\ref{lem:dualrels}.

\begin{proof}
(i)$\Leftrightarrow$(ii): If either (i) or (ii) holds $Y$ is a shared component with $X^{gh}$.  We consider the map of vector bundles $T_{X^g}|_Y\oplus T_{X^h}|_Y\to T_X|_Y$.  The kernel of this map is $T_{(X^g\cap X^h)}|_Y=T_{X^{gh}}|_Y$, and one can check locally that the cokernel is a vector bundle on $Y$.  By counting ranks, we see that transversality holds if and only if we have an equality
\[
\mathrm{rank}(T_{X^{gh}})|_Y= \mathrm{rank}(T_{X^g})|_Y+\mathrm{rank}(T_{X^h})|_Y-\mathrm{rank}(T_X)|_Y.
\]
This, in turn, occurs if and only if
\[
\begin{array}{rl}
\mathrm{codim}(X^{gh})|_Y&=\mathrm{rank}(T_X)|_Y-\mathrm{rank}(T_{X^{gh}})|_Y\\
&=\mathrm{rank}(T_X)|_Y-(\mathrm{rank}(T_{X^g})|_Y+\mathrm{rank}(T_{X^h})-\mathrm{rank}(T_X)|_Y)\\
&=\mathrm{rank}(T_X)|_Y-\mathrm{rank}(T_{X^g})|_Y+\mathrm{rank}(T_X)|_Y-\mathrm{rank}(T_{X^h})|_Y\\
&=\mathrm{codim}(X^g)|_Y+\mathrm{codim}(X^h)|_Y.
\end{array}
\]
Whence we have the desired equivalence between (i) and (ii).
\par

(ii)$\Leftrightarrow$(iv): For (iv), the vanishing of the intersection in $T_X|_Y$ is equivalent to the vanishing of the kernel $K$ in the sequence
\[
0\to K\to N_{X^g}|_Y\oplus N_{X^h}|_Y\to T_X|_Y.
\]
This property can be check complete locally at geometric points $p$ of $Y$, and hence can be reduced to the condition Lemma~\ref{lem:dualrels} (iii).  The equivalent condition of Lemma~\ref{lem:dualrels} (ii) then assures as that at each such $p$ the proposed relation of codimensions holds.  So we see that (iv) implies (ii), and we similarly find the converse implication.
\par

(iii)$\Leftrightarrow$(iv): The fact that (iii) implies (iv) is obvious, as the map $\kappa$ factors through $T_X|_Y$.  Suppose now that (iv) holds, i.e.\ that the images of the normal bundles $N_{X^g}|_Y$ and $N_{X^h}|_Y$ intersect trivially in the tangent sheaf.  Note that the image of $N_{X^g}|_Y\oplus N_{X^h}|_Y$ in $T_X|_Y$ lies in $(1-gh)T_X|_Y$ if and only if the the composition $N_{X^g}|_Y\oplus N_{X^h}|_Y\to T_X|_Y\to T_{X^{gh}}|_Y$ vanishes.  Vanishing of this composition can be checked complete locally, and hence the inclusion $N_{X^g}|_Y+ N_{X^h}\subset (1-gh)T_X|_Y$ can be checked complete locally.  Vanishing of the intersection is therefore reduced to the condition of Lemma~\ref{lem:dualrels} (iii), and the equivalent condition of Lemma~\ref{lem:dualrels} (ii) implies that $N_{X^g}|_Y+ N_{X^h}|_Y=(1-gh)T_X|_Y$.  Since these normal bundles intersect trivially in $T_X|_Y$ the sum $N_{X^g}|_Y+ N_{X^h}$ remains direct.  Finally, since $(1-gh)T_X|_Y$ is sent isomorphically to $N_{X^{gh}}|_Y$ under the projection from $T_X|_Y$, we see that $\kappa$ is an isomorphism. 
\end{proof}

\section{The Hochschild cohomology $\uHH^\bullet(\msc{X})$ as an equivariant sheaf: Following~\cite{anno}}
\label{sect:A}

We recall an argument of Anno~\cite{anno}, which produces a canonical isomorphism of equivariant sheaves
\[
\uHH^\bullet(\msc{X})\cong \oplus_{g\in G}T^{poly}_{X^g}\ot_{\O_X}\det(N_{X^g}).
\]
Note that the equivariant structure on the sum $\oplus_{g}T^{poly}_{X^g}\ot_{\O_X}\det(N_{X^g})$ is induced by the isomorphisms $h:X^g\to X^{hgh^{-1}}$.  To be precise, these isomorphisms fit into diagrams
\[
\xymatrix{
X\ar[r]^h & X\\
X^g\ar[u]^{\mrm{incl}}\ar[r]^h & X^{hgh^{-1}}\ar[u]_{\mrm{incl}},
}
\]
and we have induced isomorphisms $h_\ast T_{X^g}\cong T_{X^{hgh^{-1}}}$ and $h_\ast N_{X^g}\cong N_{X^{hgh^{-1}}}$ via conjugation by $h$, which produce the equivariant structure maps
\[
h_\ast(T^{poly}_{X^g}\ot_{\O_X}\det(N_{X^g}))\cong T^{poly}_{X^{hgh^{-1}}}\ot_{\O_X}\det(N_{X^{hgh^{-1}}}).
\]

\subsection{The Hochschild cohomology $\uHH^\bullet(\msc{X})$ as an equivariant sheaf}
\label{sect:A1}

For any $g\in G$ we have the graph $\Delta_g:X\to X\times X$ of $g$ and extensions $\sExt_{X^2}(\Delta_\ast\O_X,(\Delta_g)_\ast\O_X)$.  We consider the projection $(\Delta_g)_\ast\O_X\to \Delta_\ast\O_{X^g}$, which is induced by the projection $\O_X\to \O_{X^g}$ over $X$ and the equality $(\Delta_g)_\ast\O_{X^g}=\Delta_\ast\O_{X^g}$.  This provides a map on cohomology
\[
\sExt_{X^2}(\Delta_\ast\O_X,(\Delta_g)_\ast\O_X)\to \sExt_{X^2}(\Delta_\ast\O_X,\Delta_\ast\O_{X^g})=\sExt_{X^2}(\Delta_\ast\O_X,\Delta_\ast\O_X)|_{X^g}=T^{poly}_X|_{X^g}.
\]
\par

The identification $\sExt_{X^2}(\Delta_\ast\O_X,\Delta_\ast\O_{X^g})=\sExt_{X^2}(\Delta_\ast\O_X,\Delta_\ast\O_X)|_{X^g}$ employed above is induced by the affine local, natural, maps
\[
\begin{array}{rl}
\Ext_{\O(U^2)}(\O(U),\O(U))\ot_{\O(U)}\O(U^g)\to &H^\bullet\left(\mrm{REnd}_{\O(U^2)}(\O(U))\ot_{\O(U)}\O(U^g)\right)\\
&=H^\bullet\left(\mrm{RHom}_{\O(U^2)}(\O(U),\O(U^g)\right)\\
&=\Ext_{\O(U^2)}(\O(U),\O(U^g)).
\end{array}
\]
These local maps are isomorphisms since the cohomology $\Ext_{\O(U^2)}(\O(U),\O(U))$ is flat over $\O(U)$.  Indeed, we have the classical HKR algebra isomorphism $\uHH^\bullet(X)=\Ext_{X^2}(\Delta_\ast\O_X,\Delta_\ast\O_X)=T^{poly}_X$ induced by the identification $\uHH^1(X)=T_X$.
\par

We now have the composite map
\[
\mrm{A}_g:\sExt_{X^2}(\Delta_\ast\O_X,(\Delta_g)_\ast\O_X)\to T^{poly}_X|_{X^g}\overset{p_g}\longrightarrow T^{poly}_{X^g}\ot_{\O_X}\det(N_{X^g}),
\]
where $p_g$ is as in Section~\ref{sect:decomp}.  We recall that the determinant of the normal bundle $\det(N_{X^g})$ here lies in cohomological degree $\mrm{codim}(X^g)$.  It is shown in~\cite{anno} that each $\mrm{A}_g$ is an isomorphism.

\begin{proposition}[{\cite[Proposition 5]{anno}}]\label{prop:anno}
For each $g\in G$, the map $\mrm{A}_g$ is an isomorphism of sheaves (supported) on $X^g$.  Whence we have a canonical isomorphism
\[
\mrm{A}:\uHH^\bullet(\X)\overset{\cong}\to\bigoplus_{g\in G} T^{poly}_{X^g}\ot_{\O_X}\det(N_{X^g}).
\]
of $G$-equivariant sheaves (supported) on $X$.
\end{proposition}

We provide the proof for the reader's convenience.

\begin{proof}
As noted in~\cite{anno}, the fact that $\mrm{A}_g$ is an isomorphism can be checked complete locally at geometric points of $X^g=\mrm{Supp}\hspace{.5mm} \sExt_{X^2}(\Delta_\ast\O_X,(\Delta_g)_\ast\O_X)$.
\par

At any $g$-invariant geometric point $p$ of $X$, with corresponding maximal ideal $m_p\subset \O_{X,p}$, we choose a $g$-linear section $m_p/m_p^2\to m_p$ to deduce an isomorphism $\hat{\O}_{X,p}\cong \mrm{Sym}(m_p/m_p^2)$ of $g$-algebras.  Hence by considering the complete localization $\hat{(\mrm{A}_g)}_p$ we reduce to the affine case, where $\mrm{A}_g$ is well-known to be an isomorphism (see~\cite{anno} or~\cite[Section 4.3]{negronwitherspoon17}).
\par

The fact that the total isomorphism $\mrm{A}$ is $G$-equivariant follows from the fact that the two constituent maps
\[
\sExt_{X^2}(\Delta_\ast\O_X,\oplus_g(\Delta_g)_\ast\O_X)\to \sExt_{X^2}(\Delta_\ast\O_X,\oplus_g\Delta_\ast\O_{X^g})=\oplus_g T^{poly}_X|_{X^g}
\]
and
\[
\coprod_g p_g:\oplus_g T^{poly}_X|_{X^g}\to \oplus_g T^{poly}_{X^g}\ot_{\O_X}\det(N_{X^g})
\]
are both $G$-equivariant (cf.~\cite[Lemma 4.3.2]{negronwitherspoon17}).
\end{proof}

As we proceed we will simply write an equality $\uHH^\bullet(\X)=\bigoplus_{g\in G} T^{poly}_{X^g}\ot_{\O_X}\det(N_{X^g})$, by abuse of notation.  Via degeneration of the local-to-global spectral sequence of Section~\ref{sect:ss} we recover the linear identification of Arinkin, C\u{a}ld\u{a}raru, and Hablicsek.

\begin{corollary}[\cite{arinkinetal14}]\label{cor:arinkinetal}
For $\msc{X}=[X/G]$ as above, there is an isomorphism of vector spaces
\[
\gr \HH^\bullet(\X)=H^\bullet\left(X,\oplus_{g\in G}\!\ T^{poly}_{X^g}\otimes_{\O_X} \det(N_{X^g})\right)^G.
\]
\end{corollary}

\section{The cup product on $\uHH^\bullet(\msc{X})$}
\label{sect:product}

We employ the identification $\mrm{A}:\uHH^\bullet(\msc{X})\cong \oplus_g T^{poly}_{X^g}\ot_{\O_X}\det(N_{X^g})$ of Proposition~\ref{prop:anno}, in conjunction with work of~\cite{SW}, in order to give a complete description of the cup product on local Hochschild cohomology.
\par

We prove below that the equivariant subsheaf $\uSA(\msc{X}):=\oplus_{g\in G}\det(N_{X^g})$ in $\uHH^\bullet(\msc{X})$ is an equivariant subalgebra.  As argued in the introduction, this subalgebra accounts for the points of $\msc{X}$ which admit automorphisms.  It is also shown that the Hochschild cohomology of $\msc{X}$ is generated by the polyvector fields $T^{poly}_X\subset \uHH^\bullet(\msc{X})$ and this distinguished subalgebra $\uSA(\msc{X})$.

\subsection{Technical points regarding complete localization and shrinking $G$}

We first recall some technical information which will be useful in our analysis. Let us consider a subgroup $G'$ in $G$, and the corresponding Hochschild cohomologies $\tilde{\pi}_\ast\sExt_{X^2}(\Delta_\ast\O_X,\Delta_\ast\O_X\rtimes G')$ and $\tilde{\pi}_\ast\sExt_{X^2}(\Delta_\ast\O_X,\Delta_\ast\O_X\rtimes G)$ as in Section~\ref{sect:SPprelim}.  The split inclusion $i:\Delta_\ast\O_X\rtimes G'\to \Delta_\ast\O_X\rtimes G$ induces an inclusion
\[
i_\ast:\tilde{\pi}_\ast\sExt_{X^2}(\Delta_\ast\O_X,\Delta_\ast\O_X\rtimes G')\to \tilde{\pi}_\ast\sExt_{X^2}(\Delta_\ast\O_X,\Delta_\ast\O_X\rtimes G),
\]
which we claim is an algebra map.
\par

In the ring theoretic setting, the fact that for any $G$-algebra $R$ the inclusion $\HH^\bullet(R,R\rtimes G')\to \HH^\bullet(R,R\rtimes G)$ is an algebra map is clear.  Indeed, this is already clear at the level of the dg algebras of Hochschild cochains $C^\bullet(R,R\rtimes G')\to C^\bullet(R,R\rtimes G)$.  Whence we see that $i_\ast$ is an algebra map by restricting to $G$-stable affines, as desired.
\par

We also recall that for coherent sheaves $M$ and $N$ on a smooth variety $Y$, and any geometric point $p$ of $Y$, the complete localization of the sheaf-extensions is canonically identified with the extensions of the complete localizations $\hat{M}_p$ and $\hat{N}_p$,
\[
\sExt_{Y}(M,N)^{\wedge}_p\underset{\mrm{canon}}\cong\Ext_{\hat{\O}_{Y,p}}(\hat{M}_p,\hat{N}_p).
\]
(See e.g.~\cite{hartshorne}.)  To be clear, our completion at $p$ here is the completion of the pullback to $Y_{\bar{k}}$.  This canonical isomorphism is one of algebras when $M=N$, and one of $\Ext_{\hat{\O}_{Y,p}}(\hat{N}_p,\hat{N}_p)-\Ext_{\hat{\O}_{Y,p}}(\hat{M}_p,\hat{M}_p)$-bimodules in general.
\par

Below, when we analyze the product on Hochschild cohomology complete locally at geometric points $p$ on the fixed spaces, we must first replace $G$ with the stabilizer $G_p$ of that point.  This step is necessary so that $G$ actually acts on the completion $\hat{\O}_{X,p}$.  The above argument assures us that there is no harm in making such local replacements.

\subsection{The action of $T^{poly}_X=\uHH^\bullet(\msc{X})_e$ and generation}

For the moment, let us view the Hochschild cohomology $\uHH^\bullet(\msc{X})$ as a sheaf on $X\times X/G$ which is supported on the diagonal.  We recall that the self-extension algebra $\sExt_{X^2}(\Delta_\ast\O_X\rtimes G,\Delta_\ast\O_X\rtimes G)$ is an $\O_{X^2}$-algebra which is supported on the union $\bigcup_{g\in G} \mrm{graph}(g)$, and hence the pushforward to $X\times X/G$ is supported on the diagonal $X\subset X\times X/G$.  In particular, the algebra structure map
\[
\O_{X\times X/G} \to \uHH^\bullet(\msc{X})
\]
factors through the projection $\O_{X\times X/G}\to \O_X$, where we abuse notation to identify $\O_X$ with its pushforward to $X\times X/G$.  So we see that $\uHH^\bullet(\msc{X})$ is naturally an $\O_X$-algebra.  This is also clear from the description of the multiplication on $\uHH^\bullet(\msc{X})$ given at~\eqref{e:exp-mult}.  (The point here is that $\O_X$ is in the center of $\uHH^\bullet(\msc{X})$.)
\par

By checking locally on $G$-stable affines, one readily verifies that the algebra
\[
\uHH^\bullet(\X)={\bigoplus}_{g\in G} T^{poly}_{X^g}\ot_{\O_X}\det(N_{X^g})
\]
is graded by $G$ in the obvious way.  That is to say, the multiplication on cohomology restricts to a map
\[
\uHH^\bullet(\msc{X})_h\ot_{\O_X}\uHH^\bullet(\msc{X})_g\to \uHH^\bullet(\msc{X})_{hg}
\]
for each $g,h\in G$, where $\uHH^\bullet(\msc{X})_\tau$ is the summand $T^{poly}_{X^\tau}\ot_{\O_X}\det(N_{X^\tau})$.  In particular, $\uHH^\bullet(\msc{X})_e$ is a subalgebra in the local Hochschild cohomology.  The following lemma is easily seen by checking over $G$-stable affines.

\begin{lemma}
The identification $T^{poly}_X=\uHH^\bullet(\msc{X})_e$, from Proposition~\ref{prop:anno}, is one of equivariant sheaves of algebras.
\end{lemma}

Let us consider the case where $h$, or $g$, is the identity.  Here we have the inclusion $\mrm{det}(N_{X^g})\subset \uHH^\bullet(\msc{X})_g$ and the identification $T^{poly}_X=\uHH^\bullet(\msc{X})_e$, from which we restrict the multiplication to get a map
\[
T^{poly}_X\ot_{\O_X}\det(N_{X^g})\to T^{poly}_{X^g}\ot_{\O_X}\det(N_{X^g}).
\]
Note that for any $g\in G$ we have an algebra projection
\begin{equation}\label{eq:1024}
T^{poly}_X\to T^{poly}_X|_{X^g}\to T^{poly}_{X^g}
\end{equation}
given by the counit of the pullback functor with the algebra map $T^{poly}_X|_{X^g}\to T^{poly}_{X^g}$ induced by the spitting $T_X|_{X^g}\cong T_{X^g}\oplus N_{X^g}$ of Section~\ref{sect:decomp}.

\begin{theorem}\label{thm:Tact}
The polyvector fields $T^{poly}_X$ are a (graded) central subalgebra in local Hochschild cohomology, so that $\uHH^\bullet(\msc{X})$ is a $T^{poly}_X$-algebra.  Furthermore, for each $g\in G$, the restriction of the multiplication
\[
\uHH^\bullet(\msc{X})_e\ot_{\O_X}\det(N_{X^g})=T^{poly}_X\ot_{\O_X}\det(N_{X^g})\to T^{poly}_{X^g}\ot_{\O_X}\det(N_{X^g})=\uHH^\bullet(\msc{X})_g
\]
is the tensor product of the surjection~\eqref{eq:1024} with the identity on $\det(N_{X^g})$.
\end{theorem}

\begin{proof}
%One can see locally, over $G$-stable affines, that the multiplication on $\uHH^\bullet(\msc{X})_e$ is the usual multiplication on polyvector fields.  So the identification $\uHH^\bullet(\msc{X})$ is appropriate.
%\par

The fact that the multiplication factors through the reduction $T^{poly}_X\to T^{poly}_X|_{X^g}$, in the first coordinate, simply follows from the fact that the target sheaf is supported on $X^g$.  Recall that the kernel of the projection $T^{poly}_X|_{X^g}\to T^{poly}_{X^g}$ is the ideal $\msc{I}$ generated by $N_{X^g}=(1-g)T_X|_{X^g}$.  So we seek to show first that the multiplication map $T^{poly}_X|_{X^g}\ot_{\O_X}\det(N_{X^g})\to T^{poly}_{X^g}\ot_{\O_X}\det(N_{X^g})$ vanishes on $\msc{I}\ot\det(N_{X^g})$, and that the restriction of the multiplication
\[
T^{poly}_{X^g}\ot_{\O_X}\det(N_{X^g})\overset{\mrm{incl}}\longrightarrow T^{poly}_X|_{X^g}\ot_{\O_X}\det(N_{X^g})\overset{\mrm{mult}}\longrightarrow T^{poly}_{X^g}\ot_{\O_X}\det(N_{X^g})
\]
is the identity map.  Both of these properties can be checked complete locally at points on $X^g$.  Whence we may reduce to the affine case, as in the proof of Proposition~\ref{prop:anno}.  Both the vanishing of the product on $\msc{I}\ot_{\O_X}\det(N_{X^g})$ and the fact that the above composite is the identity are understood completely in the affine case~\cite{anno,SW}.  Centrality of $T^{poly}_X$ in $\uHH^\bullet(\msc{X})$ also follows by a complete local reduction to the linear setting.
\end{proof}

\begin{remark}
Theorem~\ref{thm:Tact} can also be deduced from the fact that the cohomology $\uHH^\bullet(\msc{X})$ is a braided commutative algebra in the category of Yetter-Drinfeld modules for $G$.  See~\cite{negron}.
\end{remark}

As an immediate corollary to Theorem~\ref{thm:Tact} one finds

\begin{corollary}
As a $T^{poly}_X$-algebra, the Hochschild cohomology $\uHH^\bullet(\msc{X})$ is generated by the determinants of the normal bundles $\det(N_{X^g})$.
\end{corollary}

\subsection{The equivariant subalgebra $\uSA(\msc{X})$}

Recall that, since $X^g$ is sent isomorphically to $X^{hgh^{-1}}$ via the action of $h$ on $X$, the codimensions of $X^g$ and $X^{hgh^{-1}}$ agree.  Rather, $\mrm{codim}(X^g)|_p=\mrm{codim}(X^{hgh^{-1}})|_{h(p)}$ at each point $p\in X^g$.  Hence we see that the sum
\[
\uSA(\msc{X}):=\bigoplus_{g\in G} \det(N_{X^g})
\]
forms an equivariant subsheaf in $\uHH^\bullet(\msc{X})$.  Since each product $\det(N_{X^g})\ot_{\O_X}\det(N_{X^h})$ is supported on the intersection $X^g\cap X^h$, one can describe the multiplication
\[
\det(N_{X^g})\ot_{\O_X}\det(N_{X^h})\to \uHH^\bullet(\msc{X})_{gh}
\]
by considering its restriction to the components of the intersection.

\begin{theorem}\label{thm:SA}
The equivariant subsheaf $\uSA(\msc{X})$ is a subalgebra in $\uHH^\bullet(\msc{X})$.  The multiplication
\[
m_{\uSA}:\det(N_{X^g})\ot_{\O_X}\det(N_{X^h})\to \det(N_{X^{gh}})
\]
is non-vanishing only on those components $Y$ of the intersection $X^g\cap X^h$ which are shared components with $X^{gh}$, and along which $X^g$ and $X^h$ intersect transversely.  On any such components $Y$, the multiplication $m_{\underline{SA}}$ is the isomorphism induced by the canonical splitting $\kappa:N_{X^g}|_Y\oplus N_{X^h}|_Y\cong N_{X^{gh}}|_Y$ of Lemma~\ref{lem:esslem} (iii), i.e.\ $m_{\uSA}=\det(\kappa)$.
\end{theorem}

\begin{proof}
Consider a geometric point $p$ in the intersection.  By Lemma~\ref{lem:codimineqs} we have
\begin{equation}\label{eq:1085}
\mrm{codim}(X^{gh})|_p\leq \mrm{codim}(X^g)|_p+\mrm{codim}(X^h)|_p
\end{equation}
with equality holding if and only if the component $Y$ containing $p$ is shared with $X^{gh}$ and the intersection is transverse along $Y$.  So we are claiming, first, that multiplication vanishes at each point $p$ at which~\eqref{eq:1085} is a strict inequality.  At any such point $p$ we consider the complete localization $\hat{(m|_{\det(N_{X^g})\ot_{\O_X}\det(N_{X^h})})}_p$ to reduce to the affine case, where the result is known to hold~\cite[Proposition~2.5]{SW}.
\par

We note that the equality $\mrm{codim}(X^{gh})|_p=\mrm{codim}(X^g)|_p+\mrm{codim}(X^h)|_p$ implies an equality of cohomological degrees
\[
\deg\big(\det(N_{X^{gh}})\big)|_p+\deg\big(\det(N_{X^{h}})\big)|_p=\deg\big(\det(N_{X^{gh}})\big)|_p
\]
in Hochschild cohomology.  Since the cup product on $\uHH^\bullet(\msc{X})$ is graded with respect to the cohomological degree, the vanishing result we have just established implies that the product on the Hochschild cohomology restricts to a product $m_{\uSA}$ on the equivariant subsheaf $\uSA(\msc{X})$.  Whence we see that $\uSA(\msc{X})$ is an equivariant subalgebra in $\uHH^\bullet(\msc{X})$.
\par

Now, at any point $p$ at which the inequality~\eqref{eq:1085} is an equality, we again reduce to the affine case via a localization $\hat{(m_{\uSA})}_p$ to see that the multiplication is in fact induced by the canonical decomposition $N_{X^g}|_Y\oplus N_{X^h}|_Y\cong N_{X^{gh}}|_Y$ of Lemma~\ref{lem:esslem} (see the proof of~\cite[Proposition~2.5]{SW}).
\end{proof}

\begin{remark}
Our algebra $\uSA(\msc{X})$ is a sheaf-theoretic variant of the ``volume subalgebra" $A_\mrm{vol}$ of~\cite{SW}.
\end{remark}

As remarked in the introduction, we can view $\uSA(\msc{X})$ as a sheaf on $\msc{X}$, or $\msc{X}\times X/G$.  However, we are not certain that $\uSA(\msc{X})$ is an invariant of $\msc{X}$ as an orbifold.  In particular, it is not immediately obvious how one should define this sheaf of algebras for orbifolds which are not global quotients.

\begin{question}
Does the algebra $\uSA(\msc{X})$, viewed as a sheaf of algebras on $\msc{X}$, have a natural interpretation as an orbifold invariant?
\end{question}

\subsection{Global descriptions}

One can readily combine Theorems~\ref{thm:Tact} and~\ref{thm:SA} to obtain a complete description of the algebra structure for $\uHH^\bullet(\msc{X})$, as in Theorem~\ref{thm:intro} of the introduction.  We observe what degeneration of the spectral sequence of Section~\ref{ss:form-deg} gives us, when paired with Theorems~\ref{thm:Tact} and~\ref{thm:SA}.

\begin{corollary}\label{cor:arinkinetal_mult}
The linear identification
\[
\gr \HH^\bullet(\X)=H^\bullet\left(X,\oplus_{g\in G}\!\ T^{poly}_{X^g}\otimes_{\O_X} \det(N_{X^g})\right)^G.
\]
of Corollary~\ref{cor:arinkinetal} is an identification of graded algebras, where $\oplus_{g\in G}\!\ T^{poly}_{X^g}\otimes_{\O_X} \det(N_{X^g})$ has the equivariant algebra structure described in Theorems~\ref{thm:Tact} and~\ref{thm:SA}.
\end{corollary}

One can also see from Theorems~\ref{thm:Tact}/\ref{thm:SA} that the only obstruction to the Hochschild cohomology being commutative is the non-commutativity of $G$.  Namely, there is an inequality of components $\uHH^\bullet(\msc{X})_{gh}$ and $\uHH^\bullet(\msc{X})_{hg}$ for non-abelian $G$.  When $G$ is abelian, these two components agree and also $\det(N_{X^{gh}})=\det(N_{X^{hg}})$.  Hence the generating subalgebra $\uSA(\msc{X})$ is commutative in this case, and it follows that entire cohomology $\uHH^\bullet(\msc{X})$ is commutative as well.  We record this observation.

\begin{corollary}
When $G$ is abelian, the cohomology $\uHH^\bullet(\msc{X})$ is (graded) commutative.
\end{corollary}

\subsection{An alternate expression of the cup product}
\label{sect:altprod}
One can alternatively introduce a ``product" on local Hochschild cohomology via the composite
\begin{equation}\label{eq:altprod}
\xymatrix{
\uHH^\bullet(\msc{X})_g\ot_{\O_X}\uHH^\bullet(\msc{X})_h\ar[r] & (T^{poly}_X|_{X^g})\ot_{\O_X}(T^{poly}_X|_{X^h})=T^{poly}_X\ot_{\O_X}T^{poly}_X|_{(X^g\cap X^h)}\ar[dl]|{\mrm{mult}}\\
 T^{poly}_X|_{(X^g\cap X^h)}\ar[r]_(.35){P_{gh}} & T^{poly}_{X^{gh}}\ot_{\O_X}\det(N_{X^{gh}})= \uHH^\bullet(\msc{X})_{gh},
}
\end{equation}
where the first map is as in Section~\ref{sect:A1} and $P_{gh}$ is the projection $p_{gh}$ on those components $Y$ of $X^g\cap X^h$ which are shared components with $X^{gh}$, and vanishes elsewhere.  One can in fact argue that the above sequence does provide a well-defined, associative, algebra structure which agrees with the cup product.  Indeed, one can make a direct comparison between the two products via the equivalent conditions (iii) and (iv) of Lemma~\ref{lem:esslem}.  The expression of the cup product via~\eqref{eq:altprod} provides a certain (re)reading of the last sentence of~\cite{anno}, and an alternate expression of~\cite[Theorem 3]{anno}.

\section{Formality and projective Calabi--Yau quotients}
\label{sect:formality}

We say the Hochschild cohomology of a Deligne--Mumford stack $\msc{X}$ is formal, as an algebra, if there is a sequence of quasi-isomorphism of $A_\infty$-algebras
\[
\msc{E}xt_{\msc{X}^2}(\Delta_\ast\O_\msc{X},\Delta_\ast\O_\msc{X})\overset{\sim}\longrightarrow \mbb{M}_1\leftarrow \cdots\to\mbb{M}_r\overset{\sim}\longleftarrow \mrm{R}\msc{H}om_{\msc{X}^2}(\Delta_\ast\O_\msc{X},\Delta_\ast\O_\msc{X}),
\]
where the $\mbb{M}_i$ are $A_\infty$-algebras in the category of sheaves of vector spaces on (the small \'{e}tale site of) $\msc{X}$, with the tensor product over the base field.
\par

Formality for smooth varieties in characteristic zero was essentially claimed by Kontsevich~\cite[Claim 8.4]{kontsevich03}, and proved by Tamarkin for affine space~\cite{tamarkin98}.  Tamarkin's proof was extended to the non-affine setting by Calaque and Van den Bergh~\cite{calaquevandenbergh10II}.  We show here that this formality result does {\it not} extend to Deligne--Mumford stacks in general.  Namely, we prove that certain Calabi--Yau orbifolds are not formal.  Here symplectic varieties feature prominently, in examples and more in depth discussions, and we continue an analysis of symplectic quotient orbifolds in subsequent sections.

\begin{remark}
One can adopt a weaker version of formality, which proposes that a sheaf of dg algebras $\msc{R}$ on $\msc{X}$ is formal if there is an isomorphism of algebras $H^\bullet(\msc{R})\cong \msc{R}$ in the derived category of sheaves of vector spaces on $\msc{X}$ (cf.~\cite{calaquevandenbergh10}).  Our method of proof shows that this weaker form of formality for Hochschild cohomology is also obstructed in general.
\end{remark}
\begin{remark}
The formality results of~\cite{kontsevich03,tamarkin98,calaquevandenbergh10,calaquevandenbergh10II} are much stronger than what is stated here.
\end{remark}

\begin{comment}
\subsection{Some results of~\cite{bridgelandkingreid01} and~\cite{caldararuI}}

In{bridgelandkingreid01} the authors consider a global quotient stack $\msc{X}$ equipped with a crepant resolution $Y\to X/G$ of the corresponding coarse space $X/G$.  They prove the following result, among others.

\begin{theorem}[{[Corollary 1.3]{bridgelandkingreid01}}]
Suppose $X$ is a complex symplectic variety and $G$ acts by symplectic
automorphisms. Assume that $Y$ is a crepant resolution of $X/G$, and take $\msc{X}=[X/G]$. Then there is an equivalence $\Phi:\mrm{D}^b(Y)\to \mrm{D}^b(\msc{X})$ provided by an explicit integral transform.
\end{theorem}

By an integral transform, we mean a functor $\Phi$ We will also need some work of C\u{a}ld\u{a}raru{caldararuI} (see also{caldararuwillerton10,orlov03}).

\begin{theorem}[{[Theorem\ 8.1]{caldararuI}}]
Let $\msc{X}$ and $\msc{Y}$ be projective Deligne--Mumford
stack with Serre duality.  Suppose $\msc{X}$ and $\msc{Y}$ admit a derived equivalence given by an integral transform.  Then the Hochschild cohomologies admit a corresponding isomorphism $\HH^\bullet(\msc{X})\cong \HH^\bullet(\msc{Y})$ of graded rings.
\end{theorem}
\end{comment}

\subsection{Failure of formality for projective Calabi--Yau quotients}

Let us consider a Calabi--Yau triangulated category $\msc{D}$ over $k$ of dimension $d$~\cite{keller08}.  So, $\msc{D}$ is a Hom-finite triangulated category for which the shift $\Sigma^d$ provides a Serre functor.  Serre duality in this case can be expressed via a global, non-degenerate, degree $-d$, functorial pairing
\[
\mrm{tr}_{a,b}:\Ext_\msc{D}(a,b)\times \Ext_\msc{D}(b,a)\to k,
\]
for arbitrary objects $a,b$ in $\msc{D}$, and $\Ext_\msc{D}(a,b)=\oplus_{n\in \mbb{Z}}\Hom_\msc{D}(a,\Sigma^nb)$.  In particular, each self-extension algebra $\mrm{Ext}_\msc{D}(a,a)$ is a graded Frobenius algebra.
\par

For $\msc{Z}$ a proper Calabi--Yau Deligne--Mumford stack, the derived category $\mrm{D}^b(\mrm{Coh}(\msc{Z}))$ is a Calabi--Yau category of dimension $d=\mrm{dim}(\msc{Z})$.  We can construct examples of such projective Calabi--Yau $\msc{Z}$ as quotient stacks as follows: Consider a projective Calabi--Yau variety $Z$ equipped with an action of a finite group $\Pi$ such that $\omega_Z\cong \O_Z$ as a $\Pi$-equivariant sheaf.  In this case the Grothendieck-Serre duality maps $\Ext^i_Z(M,N)\to \Ext^{d-i}_Z(N,M)^\ast$, for $\Pi$-equivariant $M$ and $N$, are $\Pi$-linear, by functoriality.  Since taking $\Pi$-invariants gives equivariant maps $\Ext_Z^i(M,N)^\Pi=\Ext_{[Z/\Pi]}^i(M,N)$ we may take $\Pi$-invariants of Serre duality for $Z$ to obtain Serre duality for $\msc{Z}=[Z/\Pi]$.  Whence the quotient $\msc{Z}=[Z/\Pi]$ is Calabi--Yau.
\par

We may consider specifically $Z$ a projective symplectic variety with $\Pi$ acting by symplectic automorphisms.  In this case $\Pi$-invariance of the symplectic form implies that the isomorphism $\O_Z\cong \omega_Z$ provided by the corresponding symplectic volume form is an equivariant isomorphism. 
\par

Since the product $\msc{Z}=\msc{X}\times \msc{X}$ of a Calabi--Yau Deligne--Mumford stack $\msc{X}$ with itself is again Calabi--Yau we find the following result. 

\begin{lemma}
If $\msc{X}$ is a Deligne--Mumford stack which is proper and Calabi--Yau, then the Hochschild cohomology $\HH^\bullet(\msc{X})$ is a graded Frobenius algebra.
\end{lemma}

We can now present the desired result.

\begin{theorem}\label{thm:informal}
Suppose $X$ is a projective variety and that $G$ is a finite subgroup in $\mrm{Aut}(X)$.  Suppose additionally that the action of $G$ is not free, and that the quotient $\msc{X}=[X/G]$ is Calabi--Yau. Then the Hochschild cohomology of $\msc{X}$ is not formal.
\end{theorem}

\begin{proof}
In this case $\HH^\bullet(\msc{X})$ is Frobenius, by the previous lemma.  The fact that $G$ does not act freely implies that we have nonempty fixed spaces $X^g$ of positive codimension.  Take $\msc{E}=\msc{E}xt_{\msc{X}^2}(\Delta_\ast\O_\msc{X},\Delta_\ast\O_\msc{X})$ and $\msc{R}=\mrm{R}\msc{H}om_{\msc{X}^2}(\Delta_\ast\O_\msc{X},\Delta_\ast\O_\msc{X})$.
\par

Let us suppose, by way of contradiction, that there is an isomorphism $\msc{E}\to \msc{R}$ of algebras in the derived category of sheaves of vector spaces on $\msc{X}$.  Then we get an induced isomorphism on derived global sections, and subsequent algebra isomorphism on cohomology
\begin{equation}\label{eq:2026}
H^\bullet(\msc{X},\msc{E})=H^\bullet(X,\uHH^\bullet(\msc{X}))^G\cong \HH^\bullet(\msc{X}).
\end{equation}
(The first equality in the above equation follows from the material of Section~\ref{sect:ss}.)  We claim, however, that $H^\bullet(X,\uHH^\bullet(X))^G$ admits no graded Frobenius structure, and hence that no such isomorphism~\eqref{eq:2026} exists.
\par

Indeed, such a graded Frobenius structure would be given by a linear map from top degree
\[
f_\mrm{top}:\left(H^\bullet(X,\uHH^\bullet(\msc{X}))^G\right)^{2\mrm{dim}(X)}\to k
\]
so that the Frobenius form is the product composed with this function, $\langle a,b\rangle=f_\mrm{top}(ab)$.  But now, by our explicit description of the cohomology $\uHH^\bullet(\msc{X})$ given in Section~\ref{sect:product}, we see that $H^\bullet(X,\uHH^\bullet(\msc{X}))^G$ admits a natural algebra grading by codimensions, and the top degree occurs only in the codimension $0$ component
\[
\left(H^\bullet(X,\uHH^\bullet(\msc{X}))^G\right)^{2\mrm{dim}(X)}=H^{\dim(X)}(X,\omega^\vee_X)^G.
\]
It follows that any such pairing on $H^\bullet(X,\uHH^\bullet(\msc{X}))^G$ necessarily vanishes when restricted to any positive codimension component, and must therefore be degenerate, which is a contradiction.  So $H^\bullet(X,\uHH^\bullet(\msc{X}))^G$ admits no graded Frobenius structure, no such algebra isomorphism~\eqref{eq:2026} exists, and the Hochschild cohomology of $\msc{X}$ cannot be formal.
\end{proof}

Of course, there exist many $\msc{X}$ satisfying the hypotheses of Theorem~\ref{thm:informal}.  We can consider symplectic quotients for example.  More specifically, one can consider Kummer varieties, which are constructed as quotients of abelian surfaces by the $\mbb{Z}/2\mbb{Z}$ inversion action.  There are $16$ fixed points in this instance, and the quotient scheme has $16$ corresponding isolated singularities.

\begin{corollary}\label{cor:symp_informal}
Let $X$ be a projective symplectic variety, and let $G\subset \mrm{Sp}(X)$ be a finite subgroup for which the corresponding action is not free.  Then the Hochschild cohomology of the orbifold quotient $\msc{X}$ is not formal.
\end{corollary}

\begin{remark}
It seems likely that the conclusion of Theorem~\ref{thm:informal} holds for arbitrary proper, Calabi-Yau, Deligne-Mumford stacks $\msc{Z}$ with singular coarse space $\underline{Z}$.
\end{remark}

One observes that we have actually proved the following.

\begin{corollary}\label{cor:DM_informal}
There exist smooth Deligne--Mumford stacks for which the Hochschild cohomology is not formal.
\end{corollary}

\subsection{Alternate notions of formality and derived equivalence}

Let us say that the Hochschild cohomology for $\msc{X}$ is $\beta$-formal if the higher $A_\infty$-structure on the vector space valued cohomology $\HH^\bullet(\msc{X})$ vanishes.  This is equivalent to the statement that there exists a homotopy isomorphism, or $A_\infty$ quasi-isomorphism, $\HH^\bullet(\msc{X})\overset{\sim}\to \mrm{RHom}_{\msc{X}^2}(\Delta_\ast\O_\msc{X},\Delta_\ast\O_\msc{X})$.

It is apparent from our above examples that the usual formality property for a stack $\msc{X}$ is not a derived invariant.  Indeed, if we suppose that the coarse space $X/G$ of a projective symplectic global quotient $\msc{X}$ admits a crepant resolution $Y\to X/G$, then Bridgeland--King--Reid construct a derived equivalence $\mrm{D}^b(\msc{X})\cong \mrm{D}^b(Y)$ via an explicit integral transform.  In this case $\msc{X}$ is as in Corollary~\ref{cor:symp_informal}, and hence the Hochschild cohomology is not formal.  However, formality of Hochschild cohomology for smooth varieties tells us that the Hochschild cohomology of $Y$ is formal.
\par

On the other hand, $\beta$-formality is a derived invariant (see~\cite{caldararuI,orlov03,BFN-dag}).  This $\beta$-formality is implied by usual formality for affine varieties, but the relationship between the two notions of formality is unclear in general.  Indeed, by~\cite{calaquevandenbergh10II} the Hochschild cohomology of a smooth variety $Y$ is the $A_\infty$-algebra given as the cohomology of the derived global sections $\mrm{R}\Gamma^\bullet(Y,T^{poly}_Y)$, which is not apriori formal as a dg algebra in the category of vector spaces.  So, for example, we do not know at this moment if an affine symplectic quotient $\msc{X}$ is $\beta$-formal even though any crepant resolution $Y$ of the coarse space is formal in the usual sense.  This is because such a crepant resolution will be non-affine in general, and hence $\beta$-formality is not known for $Y$.

\section{Hochschild cohomology for Symplectic quotients}
\label{sect:symplectic}

We fix $k=\mbb{C}$ in this section.  Here we consider a symplectic variety $X$ (over $\mbb{C}$) on which a finite group $G$ acts by symplectic automorphisms.  We explain how the local Hochschild cohomology of the quotient $\msc{X}=[X/G]$ simplifies in this case.  Our presentation involves a correction to an oft-repeated error in the literature (see Section~\ref{sect:remark} and Appendix~\ref{sect:pav}).  In Section~\ref{sect:H_orb} we use our description of Hochschild cohomology to propose explicit relations between Hochschild cohomology and orbifold cohomology for symplectic quotients.

\subsection{Generalities for symplectic quotients}

Let $X$ be symplectic and $G$ be a finite group acting on $X$ via symplectic automorphisms.  We let $\omega\in \Gamma(X,\Omega^2_X)$ be the symplectic form and $\pi=\omega^{-1}\in \Gamma(X,T^2_X)$ be the corresponding Poisson structure.  We note that the fixed spaces for the $G$ action are symplectic subvarieties in $X$, and we denote the corresponding symplectic form on $X^g$ by $\omega_g$ for each $g\in G$.  We fix $d_g=\frac{1}{2}\mrm{codim}(X^g)$ and $d=\frac{1}{2}\dim(X)$.
\par

Via the decomposition $\Omega_X|_{X^g}=\Omega_{X^g}\oplus \Omega_{X^g}^\perp$, where $\Omega_{X^g}^\perp$ is the sum of all nontrivial $g$-eigensheaves, we decompose the form as $\omega|_{X^g}=\omega_g+\omega_g^\perp$ with $\omega_g$ a section of $\Omega_{X^g}$ and $\omega_g^\perp$ a section of $\Omega_{X^g}^\perp\wedge \Omega_{X^g}^\perp$.  Note that the mixed term $\omega_g^{mixed}$ over $\Gamma(X^g,\Omega_{X^g}\wedge \Omega_{X^g}^\perp)$ vanishes by $g$-invariance of the restriction $\omega|_{X^g}$.  Hence the volume form is the product $(\omega|_{X^g})^d=\binom{d}{d_g}\omega_g^{d-d_g}(\omega_g^{\perp})^{d_g}$, and we have an isomorphism
\begin{equation}\label{eq:116}
\Omega_{X^g}^{2(d-d_g)}\ot_{\O_X}\det(N_{X^g}^\vee)\to \Omega_X^{2d}|_{X^g},\ \ \omega_g^{(d-d_g)}\ot (\omega_g^{\perp})^{d_g}\mapsto \binom{d}{d_g}^{-1}(\omega|_{X^g})^d
\end{equation}
given by the multiplication on $\Omega^{poly}_X|_{X^g}$.  Since the respective volume forms trivialize $\Omega_{X^g}^{2(d-d_g)}$ and $\Omega_X^{2d}|_{X^g}$, the isomorphism~\eqref{eq:116} implies that $\det(N_{X^g}^\vee)$ is globally trivial and generated by the relative volume form $(\omega_g^{\perp})^{d_g}$.  A similar analysis holds for the tangent sheaf, with $\det(N_{X^g})$ trivialized by the relative dual volume form $(\pi_g^\perp)^{d_g}$.  We normalize and take
\[
\psi_g:=\frac{1}{d_g!}(\pi_g^\perp)^{d_g}.
\]
We have just established

\begin{lemma}
Let $X$ be a symplectic variety and $G$ be a finite group acting by symplectic automorphisms on $X$.  Then the determinants of the normal bundles $\det(N_{X^g})$, for each $g\in G$, are globally trivial and generated by the global polyvector field $\psi_g$.
\end{lemma}

Recall that the subalgebra $\uSA(\msc{X})$ is generated by the determinants $\det(N_{X^g})$ as an $\O_X$-algebra, and each determinant is generated by the global polyvectors $\psi_g$.  Hence $\uSA(\msc{X})$ is generated as an $\O_X$-algebra by the sections $\psi_g$, for $g$ not the identity.  By the $G$-grading on $\uSA(\msc{X})$ we have
\begin{equation}\label{eq:126}
\psi_g\cdot \psi_h|_Y=a(g,h)\psi_{gh}|_Y,
\end{equation}
where the $Y$ range over shared components of $X^{gh}$ and $X^g\cap X^h$ and $a(g,h)$ is a section of the sheaf of units $(\prod_Y\O_Y)^\times$.
\par

By checking complete locally at closed points, to reduce to the affine case, one finds that the coefficients $a(g,h)$ are in the constant sheaf of units $\underline{\mbb{C}}$ on $\coprod_\mrm{shared\ cmpts} Y$.  Whence we see that the multiplication on $\uSA(\msc{X})$, and the Hochschild cohomology in its entirely $\uHH^\bullet(\msc{X})$, is determined by the coefficients $a(g,h)$.  We will show that the function $a$ is, essentially, a boundary.  (We use the word essentially here because $a$ is not actually a $2$-cocycle for $G$.)

\subsection{Hochschild cohomology for symplectic quotients}
\label{sect:scaling}

Consider the product $\mbb{C}^{\pi_0(X^g)}$ as the global sections of the constant sheaf $\underline{\mbb{C}}$ on $X^g$.  We seek invertible scalars $\lambda_g\in \mbb{C}^{\pi_0(X^g)}$ for each $g\in G$ such that
\begin{equation}\label{eq:137}
(\lambda_g\lambda_h)^{-1}|_Y(\lambda_{gh}|_Y)=a(g,h)|_Y\ \ \text{on shared components $Y$ between $X^{gh}$ and }X^g\cap X^h,
\end{equation}
which satisfy the $G$-equivariance condition:
\begin{equation}\label{eq:137-eq}
\text{For every component
$Z \subseteq X^g$ and every $h \in G$, } \lambda_g|_Z = \lambda_{hgh^{-1}}|_{h(Z)}.
\end{equation}
In this case the scaled generators $\mu_g=\lambda_g\psi_g$ satisfy $\mu_g\mu_h|_Y=\mu_{gh}|_Y$, and under the $\O_X$-linear isomorphism
\[
\oplus_{g\in G}\O_{X^g}\overset{\sim}\longrightarrow \uSA(\msc{X}),\ \ 1_g\mapsto \lambda_g\psi_g,
\]
the induced algebra structure on the sum $\oplus_g \O_{X^g}$ is simply given by
\begin{equation}\label{eq:transp}
f_g\cdot f_h|_Y=\left\{\begin{array}{ll}
(f_g|_Y)(f_h|_Y) & \text{if $Y$ is a shared component with }X^g\cap X^h\\
0 & \text{otherwise},
\end{array}\right.
\end{equation}
where the $f_\sigma$ are sections of their respective structure sheaves and $Y$ is an arbitrary component of $X^{gh}$ in the above formula.  We call the multiplication~\eqref{eq:transp} the {\it transparent multiplication} on the sum $\oplus_{g\in G}\O_{X^g}$.
\par

We note that the transparent multiplication on the sum $\oplus_g\O_{X^g}$ is related to the multiplication on the group ring $\O_X[G]$, although there is some nuance to consider here.  For $X=V$ a symplectic vector space, for example, and $G\subset \mrm{Sp}(V)$, the fiber of the transparent algebra $\oplus_g\O_{X^g}$ at $0$ is the associated graded algebra $\gr_F\mbb{C}[G]$ of the group algebra under the codimension filtration $F_{-i}\mbb{C}[G]=\mbb{C}\{g\in G:\mrm{codim}(V^g)\geq i\}$ (see e.g.~\cite{etingofginzburg02}).

We will see below that the determinants of the operators $(1-g|_{N_{X^g}})$ on the normal bundles $N_{X^g}$, for each $g$, provide the desired bounding function as in~\eqref{eq:137}.

\begin{lemma}\label{lem:182}
For each $g\in G$, the determinant of the operator $(1-g|_{N_{X^g}})$ is a real, positive, global section of the constant sheaf of complex numbers over $X^g$.
\end{lemma}

\begin{proof}
Consider a component $Z$ of $X^g$.  The space $T_X|_{Z}$ decomposes into eigensheaves $\oplus_{\chi}T_\chi$, where $\chi$ runs over the characters of the group $\langle g|_{T_X}\rangle$, and we have $N_{Z}=\oplus_{\chi\neq id}T_\chi$.  Invariance of the non-degenerate form $\omega:T_X|_{Z}\wedge T_X|_{Z}\to \O_{Z}$ under $g$ implies that the pairing restricts to non-degenerate pairings $T_{\chi}\times T_{\chi^{-1}}\to \O_{Z}$.  In particular, the eigensheaves come in pairs of the same rank, and if we take $\zeta_i=\chi_i(g)$ we have
\[
\det(1-g|_{N_{Z}})=2^l\prod_{i=1}^N(1-\zeta_i)(1-\bar{\zeta}_i)=2^l\prod_{i=1}^N|1-\zeta_i|^2,
\]
where $\{\chi_1,\dots,\chi_N\}$ is an enumeration of the characters for which $\zeta_i=\chi_i(g)$ is in the upper half plane in $\mbb{C}$, and $l$ is the rank of the eigensheaf with eigenvalue $-1$ for $g$.
\end{proof}

We provide a detailed proof of the following proposition in the case in which $X$ is a symplectic vector space in Appendix~\ref{sect:pav}, as there are some inaccuracies in the literature which need to be corrected.  We discuss the proof for general $X$ in Section~\ref{sect:proof_disc}.

\begin{proposition}\label{prop:coboundary}
The global sections $\mrm{det}(1-g|_{N_{X^g}})^{1/2}\in \mbb{C}^{\pi_0(X^g)}$ collectively solve equation~\eqref{eq:137} and \eqref{eq:137-eq}, where $\mrm{det}(1-g|_{N_{X^g}})^{1/2}$ denotes the positive square root.  
\end{proposition}
Note that this solution is indeed invariant under conjugation by $G$, as desired.
\begin{remark} \label{r:agh-nz}
If $G$ is abelian, then it is easy to check that the cocycle $a(g,h)=1$ for all $g$ and $h$, in spite of the fact that the determinants $\mrm{det}(1-g|_{N_{X^g}})$ are not equal to one.  Thus in this case one could also use the trivial solution to \eqref{eq:137}. However, when $G$ is nonabelian, in general $a(g,h) \neq 1$. For example, if $G=S_n$ and $X=T^*\mbb{C}^n$,
%$(\mbb{C}^{n-1})$ is the cotangent bundle of the reflection representation,
 it follows from Proposition \ref{prop:coboundary} (and one can verify directly) that $a((1 \cdots k)(k \cdots m)) = \frac{m}{k(m-k)}$.
\end{remark}

From Proposition~\ref{prop:coboundary} we obtain a simple description of the multiplication of $\uSA(\msc{X})$ for arbitrary symplectic quotients.

\begin{theorem}\label{thm:SAsymp}
Suppose $X$ is a symplectic variety and that $G$ acts on $X$ by symplectic automorphisms.  Then there is a canonical isomorphism of $\O_X$-algebras $\uSA(\msc{X})\cong \oplus_{g\in G}\O_{X^g}$, where $\oplus_g\O_{X^g}$ is given the transparent multiplication~\eqref{eq:transp}.
\end{theorem}

This theorem extends to the ``sheaf-valued'' cohomology $\uHH^\bullet(\msc{X})=\oplus_{g\in G} T^{poly}_{X^g}$ in the symplectic case as well, where $\msc{X}=[X/G]$, as usual.

\begin{theorem}\label{thm:HHsymp}
There is an isomorphism of equivariant sheaves of algebras 
\[
\oplus_{g\in G}T^{poly}_{X^g}\overset{\sim}\longrightarrow\uHH^\bullet(\msc{X}),\ \ 1_g\mapsto \det(1-g|_{N_{X^g}})^{1/2}\psi_g
\]
where the sum $\oplus_g T^{poly}_{X^g}$ is given the transparent multiplication.
\end{theorem}

Note that to include the grading we need to shift the polyvector fields as $\oplus_g T^{\bullet-\mrm{codim}(X^g)}_{X^g}$.  We ignore this degree shift for the sake of notation.
\par

Let us state explicitly what we mean by the transparent multiplication here.  Let $Y$ be a shared component for $X^{gh}$ and $X^g\cap X^h$ along which the fixed spaces intersect transversely.  For $\tau=g,h,gh$, we have the projections
\[
p_\tau:T^{poly}_X|_Y\to T^{poly}_{X^\tau}|_Y
\]
with kernel generated by the normal bundle $N_{X^\tau}$.  The containments $N_{X^g}|_Y,\ N_{X^h}|_Y\subset N_{X^{gh}}|_Y$ provided by (the proof of) Lemma~\ref{lem:esslem} (iii) imply that the projection $p_{gh}$ factors through $T^{poly}_{X^g}|_Y$ and $T^{poly}_{X^h}|_Y$, and we have induced projections $T^{poly}_{X^g}|_Y\to T^{poly}_{X^{gh}}|_Y$ and $T^{poly}_{X^h}|_Y\to T^{poly}_{X^{gh}}|_Y$.  We let $p_{gh}$ denote these induced projections, by an abuse of notation.  Whence we can define an obvious product by restricting to components $Y$ of $X^{gh}$,
\begin{equation}\label{eq:Ttransp}
(\xi_g\cdot\xi_h)|_Y=\left\{\begin{array}{ll}
p_{gh}(\xi_g|_Y)p_{gh}(\xi_h|_Y) &\text{if $Y$ is a shared component}\\
0 & \text{otherwise}.
\end{array}\right.
\end{equation}
This is what we have referred to as the transparent multiplication in Theorem~\ref{thm:HHsymp}.

\subsection{The proof of Proposition~\ref{prop:coboundary}}
\label{sect:proof_disc}

\begin{proof}[Proof of Proposition~\ref{prop:coboundary}]
The identity \eqref{eq:137-eq} is clear from the definition $\lambda_g=\mrm{det}(1-g|_{N_{X^g}})$.
 To prove \eqref{eq:137}, 
%It suffices to prove the equality~\eqref{eq:137}, for $\lambda_g=\mrm{det}(1-g|_{N_{X^g}})$, 
 it suffices to work complete locally at closed points on shared
 components $Y$ of $X^{gh}$ and $X^g\cap X^h$.  At such a point $p$ on
 a shared component $Y$ we take $V=T_X|_{\hat{Y}_p}$.  We have
 $V^g=N_{X^g}|_{\hat{Y}_p}$,
 $(V^g)^\perp=(1-g)V=T_{X^g}|_{\hat{Y}_p}$, and similar equalities for
 $h$ and $gh$.  We consider these objects as modules over the complete
 local $\mbb{C}$-algebra $K=\hat{\O}_{Y,p}$.
\par

One replaces $\mbb{C}$ with $K$ and argues just as in the appendix to arrive at Proposition~\ref{prop:coboundary}.  Namely, the proofs of Lemma~\ref{lem:pfaffian} and Proposition~\ref{p1} are equally valid over $K$.  This gives
\[
a(g,h)=\left(\frac{\det(1-gh|_{N_{X^{gh}}})}{\det(1-g|_{N_{X^g}})\det(1-h|_{N_{X^h}})}\right)^{1/2}.
\]
We have seen at Lemma~\ref{lem:182} that all of the determinants in the above formula are positive real, and hence $a(g,h)$ is either the positive or negative square root on each component $Y$.  One can check the sign of $a(g,h)$ on the fiber $V\ot_K\mbb{C}$, which reduces us to an analysis of a symplectic vector space equipped with a finite symplectic group action.  So we apply Proposition~\ref{p2} directly to find that $a(g,h)$ is in fact the positive square root.
\end{proof}

\subsection{A remark on the literature}
\label{sect:remark}

It is stated in multiple references that the coefficients $a(g,h)$ are identically $1$, in the specific case of a vector space $X=V$ and $G$ a finite subgroup in $\mrm{Sp}(V)$ (e.g.~\cite[proof of Theorem 1.8]{etingofginzburg02},~\cite[\S 3.10]{alvarez02},~\cite[Theorem 4]{anno}).  However, this is not the case, as explained in Remark \ref{r:agh-nz}.  The best that one could hope for is that the coefficients $a(g,h)$ form a $G$-equivariant ``coboundary", in the sense of~\eqref{eq:137},\eqref{eq:137-eq}, which is what is shown in Proposition~\ref{prop:coboundary}.

\section{Symplectic quotients and orbifold cohomology}
\label{sect:H_orb}

As in Section~\ref{sect:symplectic}, we suppose $X$ is a symplectic variety over $\mbb{C}$ equipped with an action of a finite group $G$ by symplectic automorphisms.  We present relations between Hochschild cohomology and orbifold cohomology, motivated by findings of~\cite{ginzburgkaledin04,dolgushevetingof05,etingofoblomkov06}.  Indeed, many of the results here can be seen as globalized variants of results from~\cite{ginzburgkaledin04}, and some materials intersect non-trivially with work of Ruan~\cite{ruan00,ruan06}.

\subsection{A dg structure on polyvector fields}

Let $F_\omega:T_X\to\Omega_X$ be the isomorphism provided by the symplectic structure, and let $\{-,-\}$ denote the Poisson bracket on polyvector fields.
\par

We have
\[
F_\omega\{\pi,f\}=2F_\omega\left(\pi(-,d_f)\right)=2\pi(F_\omega^\vee-,d_f)=2\langle -,d_f\rangle:T_X\to \msc{O}_X.
\]
This final element is simply $2 d_f\in \Omega_X$.  Furthermore, one can checking complete locally to find that the sheaf isomorphism $F^{poly}_\omega$ provides an isomorphism of complexes
\begin{equation}\label{eq:226}
F^{poly}_\omega:(T_X^{poly},\frac{1}{2}\{\pi,-\})\to (\Omega_X^{poly},d_\mrm{dR}).
\end{equation}
\begin{comment}
More specifically, the elements $\xi_i:=\{\pi,x_i\}$ provide a basis for $T_X$ complete locally, in coordinates $\hat{\O}_{X,p}\cong \mbb{C}[\![x_i]\!]$.  This is simply because the $F_\omega\xi_i=2d_{x_i}\in \Omega_{\hat{X}_p}$ provide a basis for the K\"ahler differentials.  Now the fact that 
\[
\frac{1}{2}\{\pi,\xi_i\}=\frac{1}{2}\{\pi,\{\pi,x_i\}\}=\frac{1}{4}\{\{\pi,\pi\},x_i\}=0
\]
for each $i$ implies that for arbitrary polyvectors
\[
\begin{array}{rl}
F_\omega^{poly}\left(\frac{1}{2}\{\pi,\sum_Ia_I\xi_{i_1}\dots\xi_{i_l}\}\right)&=\sum_I\frac{1}{2}F_\omega^{poly}\{\pi,a_I\} F^{poly}_\omega(\xi_{i_1}\dots\xi_{i_l})\\
&=2^l\sum_I d_\mrm{dR}(a_I)d_{x_{i_1}}\dots d_{x_{i_l}}\\
&=2^l\sum_I d_\mrm{dR}(a_Id_{x_{i_1}}\dots d_{x_{i_l}})=d_\mrm{dR}\left(F_\omega^{poly}(\sum_I a_I \xi_{i_1}\dots \xi_{i_l})\right),
\end{array}
\]
as claimed.  Of course, we can repeat this at each of the fixed spaces to obtain isomorphisms $(T_{X^g}^{poly},\frac{1}{2}\{\pi_g,-\})\cong (\Omega_{X^g}^{poly},d_\mrm{dR})$.
\end{comment}

We define the operation
\[
\{\pi,-\}:\oplus_{g\in G}T^{poly}_{X^g}\to \oplus_{g\in G} T^{poly}_{X^g}
\]
formally as the sum $\oplus_{g\in G}\{\pi_g,-\}$.
\par

We consider $\oplus_{g\in G}T^{poly}_{X^g}$ as a sheaf of algebras on $X$ under the transparent multiplication~\eqref{eq:Ttransp}.  By a nontrivial, but straightforward, computation one sees that the operator $\frac{1}{2}\{\pi,-\}$ is a degree one square zero derivation on the sheaf of algebras $\oplus_g T^{poly}_{X^g}$, with the transparent multiplication.  Rather, $(\oplus_g T^{poly}_{X^g},\frac{1}{2}\{\pi,-\})$ is a dg algebra.

\subsection{Linear identifications with orbifold cohomology}

We have the isomorphism of sheaves of complexes
\begin{equation}\label{eq:463}
(\uHH^\bullet(\msc{X}),\frac{1}{2}\{\pi,-\})\cong \bigoplus_{g\in G}(\Omega_{X^g}^\bullet,d_\mrm{dR})
\end{equation}
given by $F^{poly}_\omega$, up to shifting by the codimensions.  This induces a non-trivial dg algebra structure on the sum of the de Rham complexes.  We discuss this dg structure in Section~\ref{sect:orb_prod}, but let us first record a number of observations and linear identifications.  We define the dg algebra
\[
\uHH^\bullet_\pi(\msc{X}):=(\uHH^\bullet(\msc{X}),\frac{1}{2}\{\pi,-\})
\]
(cf.~\cite[Sect.\ 6]{ginzburgkaledin04}).

\begin{lemma}[{cf.~\cite{ginzburgkaledin04}}]\label{lem:HochFG}
There is an isomorphism of graded $G$-representations
\[
\mbb{H}^\bullet(X,\uHH^\bullet_\pi(\msc{X}))\overset{\cong}\to \bigoplus_{g\in G}H^{\bullet+\mrm{codim}(X^g)}(X^g_\mrm{an},\mbb{C}),
\]
where $\mbb{H}^\bullet$ denotes the hypercohomology and $H^\bullet(Z_\mrm{an},\mbb{C})$ denotes the singular cohomology of the analytic space associated to a variety $Z$ over $\mbb{C}$.
\end{lemma}

\begin{proof}
This follows from the isomorphism~\eqref{eq:463}, Grothendieck's isomorphism between the algebraic de Rham cohomology and analytic de Rham cohomology~\cite{grothendieck66}, and the general isomorphism $H^\bullet_\mrm{dR}(Z_\mrm{an})\cong H^\bullet(Z_\mrm{an},\mbb{C})$.  Each isomorphism here is sufficiently natural to account for the action of $G$.
\end{proof}

We take invariants to arrive at an identification with the orbifold cohomology.

\begin{theorem}\label{thm:HHorb}
There is a graded isomorphism between the hypercohomology of $\uHH^\bullet_\pi(\msc{X})$, considered as a sheaf on $\msc{X}$, and the orbifold cohomology of $\msc{X}$, $\mbb{H}^\bullet(\msc{X},\uHH^\bullet_\pi(\msc{X}))\cong H^\bullet_\mrm{orb}(\msc{X},\mbb{C})$.
\end{theorem}

\begin{proof}
In the symplectic setting the ``age shift" $a(g)$ of~\cite[Definition 1.5]{fantechigottsche03} is half the codimension of $X^g$.  Hence the orbifold cohomology is given as the invariants
\[
H^\bullet_\mrm{orb}(\msc{X},\mbb{C})=\left(\oplus_g H^{\bullet+\mrm{codim}(X^g)}(X^g_\mrm{an},\mbb{C})\right)^G
\]
\cite[Definition 1.7]{fantechigottsche03}.  Now the desired isomorphism is immediate from the equality $\mbb{H}^\bullet(\msc{X},\uHH^\bullet_\pi(\msc{X}))=\mbb{H}^\bullet(X,\uHH^\bullet_\pi(\msc{X}))^G$ and Lemma~\ref{lem:HochFG}.
\end{proof}

In the particular case of projective symplectic variety $X$, we have a direct comparison between Hochschild cohomology and orbifold cohomology. 

\begin{corollary}\label{cor:501}
If $X$ is projective over $\mbb{C}$, then there is an isomorphism of graded vector spaces $\HH^\bullet(\msc{X})\cong H^\bullet_\mrm{orb}(\msc{X})$.
\end{corollary}

\begin{proof}
By the Hodge decomposition we have
\[
H^\bullet_\mrm{orb}(\msc{X})=\big(\bigoplus_g H^\bullet(X^g_\mrm{an},\mbb{C})\big)^G=\big(\bigoplus_g H^\bullet(X^g,\Omega^\bullet_{X^g})\big)^G.
\]
Whence the identification follows from the isomorphisms $T_{X^g}\cong \Omega_{X^g}$ given by the symplectic form(s) and the vector space identification $\HH^\bullet(\msc{X})=\big(\oplus_g H^\bullet(X^g,T^{poly}_{X^g})\big)^G$ of Corollary~\ref{cor:arinkinetal} (or rather~\cite{arinkinetal14}).
\end{proof}

\begin{remark}\label{rem:quant}
\begin{enumerate}
\item The isomorphism of Theorem~\ref{thm:HHorb} was obtained by Ginzburg and Kaladin in the case of a quotient $\msc{X}=[X/G]$ with $X$ affine~\cite{ginzburgkaledin04}.
\item A variant of Theorem~\ref{thm:HHorb} should hold, without the appearance of the differential $\frac{1}{2}\{\pi,-\}$, when one replaces $(\msc{X},\O_\msc{X})$ with a quantization of the symplectic structure, i.e.\ an equivariant quantization of $(X,\O_X)$.  Specifically, if one takes $\msc{A}/\mbb{C}[\![\hbar]\!]$ a quantization of the symplectic orbifold $\msc{X}$, then one expects $\HH^\bullet(\msc{A}[\hbar^{-1}])\cong H_\mrm{orb}^\bullet(\msc{X})\ot\mbb{C}(\!(\hbar)\!)$.  (Here we view $\msc{A}$ as a sheaf of rings on the \'etale site of $\msc{X}$, so that the Hochschild cohomology should implicitly involve a smash product construction for $\msc{A}$.)  This relation was proved by Dolgushev and Etingof in the affine case.  See also~\cite{ginzburgkaledin04,etingofoblomkov06}.
\item The isomorphism of Corollary~\ref{cor:501} can already be obtained from the formality result of~\cite{arinkinetal14}.
\end{enumerate}
\end{remark}

\subsection{Conjectural, stronger, relations with Fantechi-G\"ottsche cohomology and orbifold cohomology}
\label{sect:orb_prod}

We consider here the work of Fantechi and G\"ottsche~\cite{fantechigottsche03}, in the special case of a symplectic variety $X$ and a finite group $G$ acting via symplectic automorphisms.  In~\cite{fantechigottsche03}, the authors define a noncommutative algebra structure on the sum
\[
H^\bullet_\mrm{FG}(X,G)=\bigoplus_{g\in G}H^{\bullet-\mrm{codim}(X^g)}(X_\mrm{an}^g,\mbb{C}),
\]
such that the natural action of $G$ is by algebra automorphisms.  The orbifold cohomology is then recovered as the invariants
\[
H^\bullet_\mrm{orb}(\msc{X})=H^\bullet_\mrm{FG}(X,G)^G.
\]
\par

Of course, this resembles our situation for Hochschild cohomology, where we have produced an algebra structure on the sum $H^\bullet(X,\uHH^\bullet(\msc{X}))$ such that the $G$-invariants approximately recovers the Hochschild cohomology algebra $\HH^\bullet(\msc{X})$.  However, for compact $X_\mrm{an}$ both the Fantechi-G\"ottsche cohomology and orbifold cohomology admit canonical Frobenius structures~\cite{fantechigottsche03,chenruan04}.  As is argued in the proof of Theorem~\ref{thm:informal}, this Frobenius structure obstructs the existence of an algebra identification between Hochschild cohomology and Fantechi-G\"ottsche, or orbifold, cohomology.
\par

Let us define the codimension filtration on the Fantechi-G\"ottsche cohomology in the obvious way:
\[
F_i\ H^\bullet_\mrm{FG}(X,G)=\bigoplus_{\mrm{codim}(X^\sigma)\leq i}H^{\bullet-\mrm{codim}(X^\sigma)}(X_\mrm{an}^\sigma,\mbb{C}).
\]
This filtration restricts to a filtration on the invariant subalgebra $H^\bullet_\mrm{orb}(\msc{X})$.  In the following conjecture, we denote the associated graded algebra for Hochschild cohomology appearing in the local-to-global spectral sequence simply by $\gr\HH^\bullet(\msc{X})$, as in the introduction.

\begin{conjecture}\label{conj:575}
Let $X$ be symplectic and $\msc{X}$ be the quotient stack of $X$ by a finite symplectic group action.
\begin{enumerate}
\item[(i)] The vector space isomorphism of Theorem~\ref{lem:HochFG} induces algebra isomorphisms
\[
\mbb{H}^\bullet(X,\uHH^\bullet_\pi(\msc{X}))\cong \gr_FH^\bullet_\mrm{FG}(X,G)\ \ \mrm{and}\ \ \mbb{H}^\bullet(\msc{X},\uHH^\bullet_\pi(\msc{X}))\cong \gr_FH^\bullet_\mrm{orb}(\msc{X}).
\]
\item[(ii)] When $X$ is projective, there is an algebra isomorphism $\gr\HH^\bullet(\msc{X})\cong \gr_FH^\bullet_\mrm{orb}(\msc{X})$.
\end{enumerate}
\end{conjecture}

When the quotient $X$ is projective and $X/G$ has isolated singularities the identification of $\gr_FH^\bullet_\mrm{orb}(\msc{X})$ with $\gr\HH^\bullet(\msc{X})$ follows from the fact that the Hodge decomposition respects the product on cohomology~\cite[Corollary 6.15]{voisin02}.  When $X$ has fixed spaces of positive dimension an identification between the two cohomologies would already require some consideration.
\par

We note that Ginzburg and Kaledin make a related conjecture involving quantizations, as in Remark~\ref{rem:quant}.  They propose that when one employs a quantization one should be able to avoid this associated graded procedure~\cite[Conjecture 1.6/1.3]{ginzburgkaledin04}.  We repeat their conjecture here.

\begin{conjecture}[Ginzburg-Kaledin]
For an appropriate quantization $\msc{A}$ of a symplectic quotient orbifold $\msc{X}$, one has an algebra isomorphism $\HH^\bullet(\msc{A}[\hbar^{-1}])\cong H^\bullet_\mrm{orb}(\msc{X})\ot\mbb{C}(\!(\hbar)\!)$. 
\end{conjecture}

Of course, point (ii) of Conjecture~\ref{conj:575} suggests that the Hochschild cohomology may be identified with the orbifold cohomology, before taking associated graded rings in both cases.  The following is strongly related to a conjecture of Ruan, as explained below.

\begin{conjecture}[cf.~\cite{ruan00,ruan06}]\label{conj:HH_H_orb}
For $\msc{X}$ a projective symplectic quotient orbifold, there is an algebra identification between the Hochschild cohomology $\HH^\bullet(\msc{X})$ and the orbifold cohomology $H^\bullet_\mrm{orb}(\msc{X})$.
\end{conjecture}

When $X/G$ admits a crepant resolution $Y\to X/G$, Conjecture~\ref{conj:HH_H_orb} reduces to Ruan's {\it Cohomological Hyperk\"{a}hler Resolution Conjecture}~\cite{ruan00,ruan06}, in the specific instance of a projective global quotient orbifold.  The argument is as follows: In the presence of such a resolution $Y$, we have that $Y$ is projective and symplectic.  There is then an equivalence of categories $\mrm{D}^b(Y)\cong \mrm{D}^b(\msc{X})$ by a result of Bridgeland-King-Reid~\cite[Corollary 1.3]{bridgelandkingreid01}, and subsequent algebra identification of Hochschild cohomologies $\HH^\bullet(\msc{X})=\HH^\bullet(Y)$~\cite{caldararuI} (see also~\cite{caldararuwillerton10,BFN-dag}).  Now by the Hodge decomposition, formality for Hochschild cohomology~\cite{calaquevandenbergh10}, and the isomorphism $\Omega_Y\cong T_Y$ given by the symplectic form, we obtain further algebra identifications $\HH^\bullet(Y)=H^\bullet_\mrm{dR}(Y)=H^\bullet(Y_\mrm{an},\mbb{C})$.  This gives, in total, $\HH^\bullet(\msc{X})=H^\bullet(Y_\mrm{an},\mbb{C})$.  Ruan's conjecture proposes that we have $H^\bullet_\mrm{orb}(\msc{X})=H^\bullet(Y_\mrm{an},\mbb{C})$ as algebras, and we obtained the claimed reduction.
\par

We note that, in general, there will not exist a crepant resolution of
$X/G$ (a local obstruction to its existence is discussed in
\cite{Bel-sing,BS-sra2}, among other places).
% Conjecture~\ref{conj:HH_H_orb} does not propose the existence of a
% crepant resolution of the coarse space $X/G$ in general.
For select verifications of Ruan's conjecture(s) one can see~\cite{lehnsorger03,uribe05,futian17,futianvial}, for example.

\begin{remark}
The statement of Conjecture~\ref{conj:HH_H_orb} makes perfect sense for an arbitrary projective symplectic orbifold $\msc{X}$, which needn't be a global quotient in general.  It seems reasonable to propose the conjecture in this greater generality, although the statement should certainly undergo some additional scrutiny here.  Note that our reduction of Conjecture~\ref{conj:HH_H_orb} to Ruan's conjecture does not immediately apply to more general $\msc{X}$, as the results of Bridgeland-King-Reid are only given for global quotients.
\end{remark}

\begin{remark}
It is proposed in~\cite[Proposition 6.2]{ginzburgkaledin04} that Conjecture~\ref{conj:575} (i) holds for affine $X$, in which case there are no derived global sections and the orbifold cohomology is already graded by codimension.  So we have, in the affine case, a direct algebra isomorphism $\HH_\pi^\bullet(\msc{X})\cong H^\bullet_\mrm{orb}(\msc{X})$.
\end{remark}

\appendix

\section{Comparing definitions of Hochschild cohomology for algebraic stacks}
\begin{center}
{\sc by Pieter Belmans}
\end{center}
\label{sect:cat_coho}

In \cite{lowenvandenbergh05} the Hochschild cohomology of an abelian category was introduced, and in \cite{MR2238922} it was explained how this relates to the deformation theory of abelian categories. In this appendix we explain how one can relate the Hochschild cohomology of the abelian category~$\Qcoh(\msc{X})$ to the definition of Hochschild cohomology as self-extensions of the structure sheaf of the diagonal. This agreement, under the appropriate conditions, is not unsurprising, but one step of the proof depends on a not widely known result. The condition we will impose on~$\msc{X}$ is that it is perfect, we recall the definition of this notion from \cite{MR3705292}.

\begin{definition}
  Let~$\msc{X}$ be an algebraic stack. We say~$\msc{X}$ is \emph{perfect} if it has affine diagonal, $\mrm{D}_{\mathrm{qc}}(\msc{X})$ is compactly generated and the structure sheaf is a compact object.
\end{definition}

The orbifolds considered in the present work are perfect by~\cite[Corollary 9.2]{MR3705292}, or more immediately by Lemma~\ref{lem:aff_diag},~\cite[Theorem A]{MR3705292}, and~\cite[Lemma 4.5]{MR3705292}.  Because~$\mrm{D}_{\mathrm{qc}}(\msc{X})$ is compactly generated, we know by \cite[Theorem~1.2]{1405.1888v2} that~$\mrm{D}_{\mathrm{qc}}(\msc{X})\cong\mrm{D}(\Qcoh(\msc{X}))$. Moreover, as the diagonal is an affine morphism we see that~$\Delta_*{\mathcal{O}}_{\msc{X}}\cong\mrm{R}\Delta_*{\mathcal{O}}_{\msc{X}}$. Using the projection formula \cite[Corollary~4.12]{MR3705292} we see that this object indeed corresponds to the Fourier--Mukai transform associated to the identity functor. The main result of this appendix is the following proposition.

\begin{proposition}\label{prop:belmans}
  Let~$\msc{X}$ be an algebraic stack over~$\Spec k$, which is moreover perfect. There exists an isomorphism of graded~$k$-algebras
\[
    \HH^\bullet(\msc{X})\cong\HH_{\mathrm{ab}}^\bullet(\Qcoh(\msc{X})).
\]
\end{proposition}

\begin{proof}
  Recall that the Hochschild cohomology of an abelian category is defined in terms of the injective objects in its Ind-completion. For a Grothendieck abelian category we can ignore this last step, as it already has enough injectives, so on the level of Hochschild cochain complexes we have
\[
    \mathrm{C}_{\mathrm{ab}}^\bullet(\Qcoh(\msc{X}))=\mathrm{C}^\bullet(\operatorname{Inj}\Qcoh(\msc{X})),
\]
  as~$\Qcoh(\msc{X})$ is Grothendieck abelian by \cite[\href{https://stacks.math.columbia.edu/tag/0781}{tag 0781}]{stacks-project}.

  Let~$\mrm{D}_{\mathrm{dg}}(\Qcoh(\msc{X}))$ be the dg~enhancement for~$\mrm{D}(\Qcoh(\msc{X}))$ given by the dg~category of homotopy-injective complexes of quasicoherent sheaves. By \cite[Theorem~A]{1507.05509v4} this dg~enhancement is unique up to quasi-equivalence, and in particular agrees with the definition of the derived category in derived algebraic geometry. The first step is to compare the Hochschild cohomology of~$\Qcoh(\msc{X})$ with that of~$\mrm{D}_{\mathrm{dg}}(\Qcoh(\msc{X}))$. By \cite[Proposition~A.5]{MR3062749} we have an isomorphism
\[
    \mathrm{C}^\bullet(\operatorname{Inj}\Qcoh(\msc{X}))\cong\mathrm{C}^\bullet(\mrm{D}_{\mathrm{dg}}(\Qcoh(\msc{X})))
\]
  of~$\mathrm{B}_\infty$-algebras, and in particular of graded algebras when considering their cohomology.

  Now by \cite[Corollary~8.1]{MR2276263} we have an isomorphism
\[
    \mathrm{C}^\bullet(\mrm{D}_{\mathrm{dg}}(\Qcoh(\msc{X})))\cong\operatorname{\mrm{R}\mathcal{H}om}(\mrm{D}_{\mathrm{dg}}(\Qcoh(\msc{X})),\mrm{D}_{\mathrm{dg}}(\Qcoh(\msc{X})))(\mathrm{id}_{\mrm{D}_{\mathrm{dg}}(\Qcoh(\msc{X}))},\mathrm{id}_{\mrm{D}_{\mathrm{dg}}(\Qcoh(\msc{X}))})
\]
  which preserves the graded algebra structure on both sides when taking cohomology by \cite[Theorem~6.1]{MR2276263}. Recall that~$\operatorname{\mrm{R}\mathcal{H}om}$ is the internal Hom in the homotopy category of dg~categories for the model structure whose weak equivalences are the quasi-equivalences.

  By \cite[Theorem~1.2(2)]{BFN-dag} we have an equivalence
\[
    \operatorname{\mrm{R}\mathcal{H}om}_{\mathrm{cont}}(\mrm{D}_{\mathrm{dg}}(\Qcoh(\msc{X})),\mrm{D}_{\mathrm{dg}}(\Qcoh(\msc{X})))\cong\mrm{D}_{\mathrm{dg}}(\Qcoh(\msc{X}\times_k\msc{X})),
\]
  where the subscript in the left-hand side indicates we are considering the full subcategory of continuous dg~functors. As the identity functor is colimit-preserving, and is presented by the kernel~$\Delta_*\mathcal{O}_{\msc{X}}$ on the right-hand side, we obtain the desired isomorphism in cohomology.
\end{proof}

\begin{remark}
  There are some alternatives to considering the Hochschild cohomology of the abelian category~$\Qcoh(\msc{X})$ or the dg~category~$\mrm{D}_{\mathrm{dg}}(\Qcoh(\msc{X}))$. First, if~$\msc{X}$ is moreover assumed to be noetherian, then by \cite[Proposition~15.4]{MR1771927} we have that~$\operatorname{Ind}\mathrm{coh}(\msc{X})\cong\Qcoh(\msc{X})$ (as abelian categories), so~$\HH_{\mathrm{ab}}^\bullet(\mathrm{coh}(\msc{X}))\cong\HH_{\mathrm{ab}}^\bullet(\Qcoh(\msc{X}))$.

  We also have an identification~$\HH_{\mathrm{dg}}^\bullet(\mrm{D}_{\mathrm{dg}}(\Qcoh(\msc{X})))\cong\HH_{\mathrm{dg}}^\bullet(\mathrm{Perf}_{\mathrm{dg}}(\msc{X}))$. This follows from the similar observation that~$\mrm{D}_{\mathrm{dg}}(\Qcoh(\msc{X}))\cong\operatorname{Ind}\mathrm{Perf}_{\mathrm{dg}}(\msc{X})$, and the equivalence from \cite[\S3.1, \S4.1]{BFN-dag}.
\end{remark}

It would be interesting to extend other definitions of Hochschild cohomology for schemes to the setting of algebraic stacks, and understand when they agree with each other. This in particular applies to the sheafification of the Hochschild complex, as in \cite{swan96}, or a Gerstenhaber--Schack approach (e.g.~using simplicial schemes), generalising that of \cite{GS-acdt}.

\section{Some basic facts about finite group representations}
\label{sect:lin_alg}

We fix $G$ a finite group whose order is not a divisible by the
characteristic of $k$. For any subgroup $H < G$, since the order of
$H$ is not a multiple of the characteristic of $k$, we can consider
the symmetrizer element
$\int_H := |H|^{-1} \sum_{\gamma \in H} \gamma$ in the group algebra
$k[H]$. If $g \in G$, we also let $\int_g := \int_{\langle g \rangle}$
for $\langle g \rangle$ the cyclic subgroup generated by $g$.

\begin{lemma}\label{lem:gsplit}
For any $\langle g\rangle$-representation $V$, $g\in G$, the two
subspaces $V^g\to V$ and $(1-g)V\to V$ provide a canonical splitting
$V=V^g\oplus (1-g)V$.
\end{lemma}

\begin{proof}
Suppose $V$ is finite dimensional.  We have the endomorphism
$(1-g):V\to V$ with kernel $V^g$ and image $(1-g)V$.  So the sum of
the dimensions of these subspaces is equal to the dimension of $V$.
We know that for any $v\in V^g\cap(1-g)V$ invariance implies $v=\int_g
v$ and any expression $v=v'-gv'$ gives $\int_g v=\int_g v'-\int_g
gv'=0$.  So $v=0$.  Hence the intersection is trivial, and we have
$V=V^g\oplus (1-g)V$.  The result for infinite dimensional
representations follows from the fact that any such representation is
the union of its finite dimensional subrepresentations.
\end{proof}

In~\cite[Lemma 2.1]{SW} the following result is proved over
$\mathbb{C}$.  Since we would like to be able to work over a more general  base
field we provide a generalization.

\begin{lemma}\label{lem:dualrels}
  For any finite dimensional $H:=\langle g,h\rangle$-representation $V$ the
  following conditions are all equivalent:
\begin{enumerate}
\item[(i)] $(1-g)V\subset (1-gh)V$; {\rm (i')} $(1-g)V^\ast\subset (1-gh)V^\ast$;
\item[(ii)] $(1-g)V+(1-h)V=(1-gh)V$; {\rm (ii')} $(1-g)V^\ast+(1-h)V^\ast=(1-gh)V^\ast$;
\item[(iii)] $(1-g)V\cap (1-h)V= 0$ ; {\rm (iii')} $(1-g)V^\ast\cap (1-h)V^\ast= 0$;
\item[(iv)] $V^H = V^{gh}$; {\rm (iv')} $(V^*)^H = (V^*)^{gh}$.
\end{enumerate}
These are also equivalent to (a) $(1-gh) V = (1-g)V \oplus (1-h) V$,
(b) $V^{gh} = V^g \cap V^h$, (c) $V=V^g+V^h$, and the same with $V$
replaced by $V^\ast$.
\end{lemma}

\begin{proof}
  Note that $(1-\gamma)V^\ast = (V^\gamma)^\perp$, the annihilator of
  $V^\gamma$ in $V^*$.  Let $\rho: k[H] \to \End(V)$ be the
  representation, with contragredient
  $\rho^\ast(\gamma):=\rho(\gamma^{-1})^\ast$. Then $\dim V^g$ is the
  dimension of the eigenspace of one of $\rho(g)$ and $\dim
  (V^\ast)^g$ is the dimension of the eigenspace of one of
  $\rho^\ast(g) = \rho(g^{-1})^\ast$, which are the same.   
 It follows that $\dim V^H
  = \rk(\rho(\int_H)) = \rk(\rho^\ast(\int_H))=\dim (V^*)^H$ (said differently,
  by Maschke's theorem, $V^H \to V \to V_H$ is an isomorphism
  and $(V^*)^H = V_H^*$ in this case).

  Now, consider condition (iv). Since $V^H \subseteq V^{gh}$, the
  condition is equivalent to $\dim V^H = \dim V^{gh}$. By the
  preceding paragraph, this is equivalent to $\dim (V^*)^H = \dim
  (V^*)^{gh}$, i.e., to (iv').

To prove these are equivalent to (i), (i'), (ii), and (ii'),
 by the symmetry of replacing $V$ by $V^\ast$, it is
enough to show that (iv) is equivalent to (i') and also to (ii').

We claim that (i') is equivalent to (iv). Indeed,
(i') states that $(V^g)^\perp \subset (V^{gh})^\perp$, i.e., that
$V^{gh} \supset V^g$.  Since $H$ is generated by $g$ and $gh$, this is
equivalent to saying that $V^{gh} = V^H$, as desired.

Next consider (ii'). This states that
$(V^g)^\perp + (V^h)^\perp = (V^{gh})^\perp$.  Equivalently, $(V^g \cap
V^h)^\perp = (V^{gh})^\perp$.  But, since $H$ is generated by $g$ and $h$,
this states that $(V^H)^\perp = (V^{gh})^\perp$. This is equivalent to $V^H
= V^{gh}$, as desired.

Now we prove (iii) and (iii') are equivalent to each other.
Condition (iii') states that the common annihilator of $V^g$
and $V^h$ is zero, i.e., $V=V^g+V^h$. This is equivalent to $\dim V =
\dim V^g + \dim V^h - \dim (V^g \cap V^h)$. Since $H$ is generated by
$g$ and $h$, this is again equivalent to $\dim V = \dim V^g + \dim V^h
- \dim V^H$. By our dimension equalities, this is equivalent to $\dim
V^\ast = \dim (V^\ast)^g + \dim (V^\ast)^h - \dim (V^\ast)^H$, i.e.,
to (iii).

Next we prove that (iii) and (iii') together imply (iv),
via the Shepler-Witherspoon
trick.  Suppose (iii) and (iii') hold. By (iii'), we have $V =
V^g + V^h$, and by (iii), we have $(1-g)V \cap (1-h)V = 0$.
Suppose that $v \in V^{gh}$. Then write $v = v_1 + v_2$ for $v_1 \in
V^g$ and $v_2 \in V^h$. Then $gv = v_1 + gv_2$ but also $gv=h^{-1}v =
h^{-1}v_1 + v_2$, so $(1-h^{-1})v_1 = -(1-g)v_2$. By assumption, this
must be zero. Thus $v \in V^g \cap V^h = V^H$. Therefore $V^{gh} =
V^H$.

Finally, we prove that (ii) and (iv) imply (iii').  Suppose (ii) and
(iv). By (ii), we have $(1-g)V + (1-h)V = (1-gh)V$, so $\dim V - \dim
V^g + \dim V - \dim V^h = \dim V - \dim V^{gh}$. By (iv), the RHS
equals $\dim V - \dim V^H = \dim V - (\dim V^g \cap \dim V^h)$. We
obtain $\dim V^g + \dim V^h = \dim V - \dim (V^g \cap V^h)$, i.e., $V
= V^g + V^h$. As we saw already, this is equivalent to (iii').

For the final statement, it is clear that (ii) and (iii) together are
equivalent to (a), and that (iv) is equivalent to (b).  We already saw
in the proof that (iii) is equivalent to (c).
\end{proof}

\section{Actions on symplectic vector spaces}\label{sect:pav}
\begin{center}
{\sc by Pavel Etingof, Cris Negron, and Travis Schedler}
\end{center}

\subsection{The main theorem}
Let $V$ be a finite dimensional symplectic vector space over $\Bbb C$ with a nondegenerate element $\pi\in \wedge^2V$ (defining a symplectic form on $V$). 
Let $g\in {\rm Sp}(V)$ be an element contained in a compact (e.g., finite) subgroup. To ease notation, we adopt a subscript notation $V_g$ for the subspace of $g$-invariants in $V$, and let $V_g^\perp\subset V^*$ be its orthogonal complement. Note that under the isomorphism $\pi: V^*\cong V$, the space $V_g^\perp$ gets identified with ${\rm Im}(1-g)$.  Take $d_g=\mrm{dim}(V_g^\perp)/2$.

Let $\pi_g^\perp$ be the restriction of $\pi$ to $V_g^\perp$ (a symplectic form), and let $\psi_g=(d_g!)^{-1}(\pi_g^\perp)^{d_g}$ be the top exterior power of $\pi_g^\perp$ (a volume form on $V_g^\perp$). 
Let  
$$
\mu_g:={\rm det}(1-g|_{V_g^\perp})^{1/2}\psi_g.
$$
Note that ${\rm det}(1-g|_{V_g^\perp})>0$, since the eigenvalues of $g$ have absolute value $1$ and for each eigenvalue $\lambda$ we have an eigenvalue $\lambda^{-1}$. Thus, the square root is well defined.

Now let $g,h\in {\rm Sp}(V)$ be contained in a compact subgroup. It is clear that $V_{gh}\supseteq V_g\cap V_h$, so $V_{gh}^\perp\subseteq V_g^\perp+V_h^\perp$. 

Assume that $V_g+V_h=V$. By Lemma~\ref{lem:dualrels}, in this case $V_{gh}=V_g\cap V_h$, hence $V_{gh}^\perp=V_g^\perp\oplus V_h^\perp$. 
\begin{comment}
Indeed, by replacing 
$V$ by the orthogonal complement of $V_g\cap V_h$ under the symplectic form, we may assume that $V_g\cap V_h=0$, i.e., $V=V_g\oplus V_h$. 
By taking orthogonal complements under the symplectic form, this implies that ${\rm Im}(1-g)\oplus {\rm Im}(1-h)=V$.  
Now, if $v\in V_{gh}$, we have $hv=g^{-1}v$, hence $(1-h)v=(1-g^{-1})v=0$. Thus $v\in V_g\cap V_h=0$, i.e. $V_{gh}=0$, as desired.
\end{comment}
Our main result is the following theorem. 

\begin{theorem}\label{mainthe}
If $V_g+V_h=V$ then 
$$
\mu_{gh}=\mu_g\wedge \mu_h
$$
(where we interpret $\mu_g,\mu_h$ as exterior forms on $V_{gh}^\perp$ using its direct sum decomposition). 
\end{theorem} 

This corrects an error in the proof of Theorem 1.8(i) in \cite{etingofginzburg02}, where it is erroneously claimed in formula (2.11) that $\psi_g\wedge \psi_h=\psi_{gh}$ (i.e., the necessary normalization of $\psi_g$ by $\det(1-g|_{V_g^\perp})^{1/2}$ is missing). This omission already seems to occur in a previous paper of Alvarez, \cite{alvarez02}, where this result is proved for $c=0$, i.e., for the semidirect product of a finite group with a Weyl algebra, see \cite{alvarez02}, 3.9. Namely, the proof of Theorem 1.8(i) of \cite{etingofginzburg02} becomes valid if $\psi_g$ is replaced with $\mu_g$ (and the same applies to the  proof of the main result of \cite{alvarez02}). 

\subsection{Proof of Theorem~\ref{mainthe}} 

Let $V$ be a finite dimensional complex vector space with a symplectic structure defined by a nondegenerate element $\pi\in \wedge^2V$ (i.e., the symplectic form is 
$\omega=\pi^{-1}$). Let $W,U$ be nondegenerate subspaces 
of $V$ of dimensions $2r,2s$, and suppose that $V=W\oplus U$. Let $\psi_W, \psi_U, \psi_V=\frac{1}{(r+s)!}\pi^{r+s}$ be the canonical elements in $\wedge^{2r} W, \wedge^{2s} U, \wedge^{2r+2s}V$ defined by $\pi$. Then we may naturally regard $\psi_W,\psi_U$ as elements of $\wedge V$, and 
$$
\psi_W\wedge \psi_U=a(W,U)\psi_V, 
$$
where $a(W,U)\in \Bbb C$. 

Let us compute $a(W,U)$. To this end, choose a symplectic basis $v_i$ of $V$ (i.e., $\pi=\sum_{p=1}^{r+s} v_{2p-1}\wedge v_{2p}$), 
and assume that $v_{2r+1},...,v_{2r+2s}$ is a symplectic basis of $U$. Let $b_{ij}, 1\le i\le 2r, 1\le j\le 2s$ be the unique numbers such that the vectors 
$$
w_i:=v_i+\sum_{j=1}^{2s} b_{ij}v_{r+j}
$$
belong to $W$. Let $\omega(w_i,w_k)=c_{ik}$ and $(c_{ik})=C$. This is clearly an invertible skew-symmetric matrix. We have 
$$
w_1\wedge...\wedge w_{2r}=v_1\wedge...\wedge v_{2r}+...,   
$$
where $...$ means terms containing $v_{2r+j}$, $j=1,...,2s$. Hence 
\[
w_1\wedge...\wedge w_{2r}\wedge \psi_U=\psi_V.
\]
At the same time, it is clear that ${\rm Pf}(C^{-1})w_1\wedge...\wedge w_{2r}=\psi_W$, where ${\rm Pf}$ is the Pfaffian. 
Hence, we get 

\begin{lemma}\label{lem:pfaffian}
$$
a(W,U)={\rm Pf}(C^{-1}).
$$
\end{lemma} 

Now let $P_W, P_U: V\to V$ be the orthogonal projectors onto $W, U$ (with respect to the symplectic form).

\begin{lemma}\label{le1} (i) We have 
$$
\det(P_W-P_U)=\det(C^{-1}). 
$$

(ii) Let $X: W\to W$ and $Y: U\to U$ be invertible linear operators. 
Then 
$$
\frac{\det(XP_W-YP_U)}{\det(X)\det(Y)}=\det(C^{-1}).
$$
\end{lemma}  

\begin{proof} 
(i) Let $J_n$ be the $2n$ by $2n$ matrix of the standard symplectic form (so $J_n^{-1}=-J_n$), and $B=(b_{ij})$.
Then it is easy to see that $C=J_r+BJ_sB^T$.

Since the projectors $P_W,P_U$ are self-adjoint with respect to the symplectic form, we have 
$$\omega((P_W-P_U)x,y)=\omega(P_Wx,y)-\omega(x,P_Uy).$$ Hence for $i\le 2r$ we have 
$$
\omega((P_W-P_U)w_i,v_k)=\omega(w_i,v_k)-\omega(w_i,P_Uv_k)=\omega(w_i,v_k)=\omega(v_i,v_k)
$$
if $k\le 2r$, and 
$$
\omega((P_W-P_U)w_i,v_k)=\omega(w_i,v_k)+\omega(w_i,P_Uv_k)=0
$$
if $k>2r$. 
Also for $j,l=1,...,2s$ we have 
$$
\omega((P_W-P_U)v_{j+2r},v_{l+2r})=\omega(P_Wv_{j+2r},v_{l+2r})-\omega(v_{j+2r},v_{l+2r})=d_{jl},
$$
where $(d_{jl})=D=-B^TC^{-1}B-J_s$.   
Thus, 
$$
\det(P_W-P_U)=\det(D)=\det(1-B^TC^{-1}BJ_s)=\det(1-C^{-1}BJ_sB^T)=
$$
$$
\det(1-C^{-1}(C-J_r))=\det(C^{-1}J_r)=\det(C^{-1}),
$$
as desired. 

(ii) follows from (i), since $XP_W-YP_U=\mrm{diag}(X,Y)(P_W-P_U)$. 
\end{proof}

Now let $g,h: V\to V$ be semisimple symplectic transformations such that $W={\rm Im}(1-g)$ and $U={\rm Im}(1-h)$. 
We can pick $X=1-g|_W$ and $Y=1-h^{-1}|_U$, then part (ii) of Lemma \ref{le1} gives: 
$$
\det(C^{-1})=\frac{\det(h^{-1}-g)}{\det(1-g|_V)\det(1-h^{-1}|_U)}=\frac{\det_V(1-gh)}{\det_W(1-g)\det_U(1-h)}.
$$

Combining the above results, we obtain 

\begin{proposition}\label{p1} We have
$$
a(W,U)=\left(\frac{\det_V(1-gh)}{\det_W(1-g)\det_U(1-h)}\right)^{1/2}.
$$
In particular, the branch of the square root on the right hand side is well defined by the condition that $a(W,U)=1$ when $gh=hg$ (i.e., 
$W$ and $U$ are orthogonal). 
\end{proposition} 

We will also need 

\begin{proposition} \label{p2} Let $g,h$ be contained in a compact subgroup of ${\rm Sp}(V)$. Then $a(W,U)>0$ (i.e., we need to choose the positive value of the square root). 
\end{proposition} 

\begin{proof}  Our job is to show that ${\rm Pf}(C^{-1})>0$. Let $K\subset Sp(V)$ be a maximal compact subgroup containing $g,h$. 
Recall that all maximal compact subgroups of ${\rm Sp}(V)$ are conjugate to $U(n,\Bbb H)$, the unitary group over the quaternions, where $n=r+s={\rm dim}(V)/2$ and $V=\Bbb H^n$. 
If $g,h\in U(n,\Bbb H)$, then ${\rm Im}(1-g)$ and ${\rm Im}(1-h)$ are quaternionic subspaces of $V$.

To prove the positivity of ${\rm Pf}(C^{-1})$, it suffices to show that the space $T$ of pairs $g,h\in U(n,\Bbb H)$ with ${\rm Im}(1-g)$, ${\rm Im}(1-h)$ 
complementary of dimensions $r,s$ is connected, since when $g$ and $h$ commute, the Pfaffian is equal to $1$.

To this end, first note that any quaternionic subspace $U$ of $V$ of dimension $r$ is automatically nondegenerate, as it is conjugate under $U(n,\Bbb H)$ to the standard subspace 
$\Bbb H^r\subset \Bbb H^n$. Consider the space of pairs of complementary quaternionic subspaces $W,U$ with ${\rm dim}_{\Bbb H}(W)=r$. This is an open subset $Y$ of 
${\rm Gr}(r,\Bbb H^n)\times {\rm Gr}(s,\Bbb H^n)$ (where ${\rm Gr}(r,\Bbb H^n)$ is the Grassmannian of quaternionic subspaces in $\Bbb H^n$ of dimension $r$)  with complement of real codimension $4$ (codimension $1$ over $\Bbb H$). Indeed, we have a fibration 
$f: Y\to {\rm Gr}(r,\Bbb H^n)$, such that $f^{-1}(P)$ for $P\in {\rm Gr}(r,\Bbb H^n)$  is (a translate of) the big (quaternionic) Schubert cell in ${\rm Gr}(s,\Bbb H^n)$, hence the complement is the union of lower dimensional cells, which are affine spaces over $\Bbb H$. Hence $Y$ is connected.

Now, for each $(W,U)\in Y$, the set of possible $(g,h)$ giving this pair is $U_*(r,\Bbb H)\times U_*(s,\Bbb H)$, where
$U_*(r,\Bbb H)$ is the set of elements of $U(r,\Bbb H)$ without eigenvalue $1$. So it suffices to show that $U_*(r,\Bbb H)$ is connected. 
But it is well known that for simple compact Lie group, the set of non-regular elements has codimension $3$. In particular, since $U(r,\Bbb H)$ is simple,
the complement of $U_*(r,\Bbb H)$ in $U(r,\Bbb H)$ has codimension $3$, which implies that $U_*(r,H)$ is connected. Since $T$ fibers over $Y$ with fibers
 $U_*(r,\Bbb H)\times U_*(s,\Bbb H)$, it is connected, as desired. 
\end{proof} 

Now, Theorem \ref{mainthe} follows from Propositions \ref{p1}, \ref{p2}  in the case $V_{gh}=0$, and the general case easily follows from this case by replacing 
$V$ by the orthogonal complement of $V_{gh}$.

\bibliographystyle{abbrv}
%\bibliography{sba}

\def\cprime{$'$}

\end{document}